\documentclass{article}

\usepackage[margin=2.5cm]{geometry}

\usepackage{lipsum}
\usepackage{amsfonts}
\usepackage{graphicx}
\usepackage{epstopdf}
\usepackage{algorithmic}
\usepackage{amsmath}
\usepackage{amssymb}
\usepackage{amsthm}
\usepackage{mathrsfs}
\usepackage{dsfont}
\usepackage[section]{algorithm}
\usepackage[labelfont=up]{subcaption}
\usepackage{afterpage}
\usepackage[title]{appendix}

\usepackage[colorlinks=true,hyperindex=true]{hyperref}
\hypersetup{
    colorlinks,
    citecolor=black,
    linkcolor=black,
    urlcolor=black,
    pdfnewwindow=true,
    pdfstartview={FitH}
}

\DeclareMathOperator{\Tr}{Tr}
\renewcommand\d{\mathrm{d}}
\newcommand{\bR}{\ensuremath{\mathbb R}}
\newcommand\sA{\mathscr{A}}
\newcommand\cA{\mathcal{A}}
\newcommand\cL{\mathcal{L}}
\newcommand\1{\mathds{1}}
\newcommand\E{\mathbb{E}}
\def\epsilon{\varepsilon}
\newcommand\cP{\mathcal{P}}
\newcommand\wP{\widetilde{P}}
\newcommand\hP{\widehat{P}}
\newcommand\wLambda{\widetilde{\Lambda}}
\newcommand\hLambda{\widehat{\Lambda}}
\newcommand\hlambda{\widehat{\lambda}}
\newcommand\hmu{\widehat{\mu}}
\newcommand{\pair}[2]{(#1,#2)}

\newtheorem{theorem}{Theorem}[section]
\newtheorem{assumption}{Assumption}[section]
\newtheorem{proposition}[theorem]{Proposition}
\newtheorem{example}{Example}[section]
\newtheorem{remark}{Remark}[section]

\numberwithin{equation}{section}

\title{Computing Large Deviation Rate Functions of Entropy Production for Diffusion Processes by an Interacting Particle Method}
\author{Zhizhang Wu
\!\footnote{Department of Mathematics, The University of Hong Kong, Pokfulam Road, Hong Kong SAR, China. wzz14@tsinghua.org.cn}
\ , 
Renaud Raqu\'epas
\!\footnote{Courant Institute, New York University, New York, NY 10012, United States. rr4374@nyu.edu}
\ , 
{Jack Xin}
\!\footnote{Department of Mathematics, University of California at Irvine, Irvine, CA 92697, United States. jack.xin@uci.edu}
\ , and 
Zhiwen Zhang
\!\footnote{Corresponding author. Department of Mathematics, The University of Hong Kong, Pokfulam Road, Hong Kong SAR, China. Materials Innovation Institute for Life Sciences and Energy (MILES), HKU-SIRI, Shenzhen, China. zhangzw@hku.hk}}

\date{}

\begin{document}

\maketitle

\begin{abstract}
We develop an interacting particle method (IPM) for computing the large deviation rate function of entropy production for diffusion processes, with emphasis on the vanishing-noise limit and high dimensions. The crucial ingredient to obtain the rate function is the computation of the principal eigenvalue $\lambda$ of elliptic, non-self-adjoint operators. We show that this principal eigenvalue can be approximated in terms of the spectral radius of a discretized evolution operator, which is obtained from an operator splitting scheme and an Euler--Maruyama scheme with a small time step size. We also show that this spectral radius can be accessed through a large number of iterations of this discretized semigroup, which is suitable for computation using the IPM. The IPM applies naturally to problems in unbounded domains and scales easily to high dimensions. We show numerical examples of dimensions up to 16, and the results show that our numerical approximation of $\lambda$ converges to the analytical vanishing-noise limit within visual tolerance with a fixed number of particles and a fixed time step size. It is numerically shown that the IPM can adapt to singular behaviors in the vanishing-noise limit. We also apply the IPM to explore situations with no explicit formulas of the vanishing-noise limit. Our paper appears to be the first one to obtain numerical results of principal eigenvalue problems for non-self-adjoint operators in such high dimensions.


\end{abstract}

\section{Introduction}
\label{sec: intro}

The problem we are interested in concerns the time reversibility of diffusion processes, as famously studied
by Kolmogorov as early as 1937 \cite{K37}. He found among other things that, with $V$ a smooth potential function and $b$ a non-conservative smooth vector field,
stochastic differential equations (SDEs) in $\mathbb{R}^d$ of the form
\begin{equation}
    \begin{cases}
        \d X_t = - \nabla V(X_t) \, \d t +\, b(X_t)\, \d t + \sqrt{2 \varepsilon} \, \d B_t, \\
        \phantom{d}X_0 \sim \mu
    \end{cases}
    \label{SDE1}
\end{equation}
are invariant under time reversal only when $b = 0$ and the density of the initial measure $\mu$ is proportional to $\exp(-\varepsilon^{-1}V)$; see Section \ref{sec: Feynman--Kac semigroups} for precise statements of the assumptions on $V$ and $b$ we will work with.
While time-reversed models have a long history of applications to fields such as signal processing \cite{ljung1976backwards,sidhu1976new} and electric circuit theories \cite{anderson1979forwards,anderson1979passive}, and have been adopted in recent years as a way to generate high-quality images
in computer vision \cite{SoDiff2005,DDM2020}, we will focus on questions from stochastic thermodynamics.
When the time reversal of a diffusion process is still a diffusion process \cite{And_80,haussmann1986time}, a natural question is how distinguishable the two processes are, i.e.\ how irreversible the original diffusion is.
One classical way to quantify irreversibility is to compute an observable called entropy production.
In the large-time limit or the steady-state regime, the entropy production for \eqref{SDE1} can be computed through the Clausius-like entropy (Stratonovich) integral
\begin{equation}
\label{eq:def-Seps}
    S^\varepsilon_t = \frac 1{\varepsilon} \, \int_{0}^{t}\, \langle \, b(X_s), \circ \, \d X_s \rangle,
\end{equation}
which in the language of statistical thermodynamics is the work done by the non-conservative part of
the drift force in \eqref{SDE1}, rescaled by temperature \cite{Ku98,LS99}. Here, the definition and physical interpretation of the entropy production \eqref{eq:def-Seps} for \eqref{SDE1} rely on the interpretation as a small-mass approximation (a.k.a.\ Kramers--Smoluchowski limit); it should be adapted naturally in the presence of momentum variables, which should change sign under time reversal; see e.g. \cite{EPRB99,JPS17,LS99}. We refer the readers to \cite{DZ23} for a discussion of other decompositions of the drift force and to \cite{JPS17,raquepas2020large} for a rigorous comparison with other measures of irreversibility, including the point of view of hypothesis testing of the arrow of time.
The study of these different notions of entropy production\,---\,and more precisely of their large deviations\,---\,has driven important theoretical progress in non-equilibrium statistical physics since the 1990s; see e.g. \cite{Cr99,ECM93,ES94,GC95,Ku98,LS99,vZC03}. One key feature of the theory of entropy production is that the positivity of the mean entropy production rate is considered as a key signature of steady non-equilibrium phenomena.

Let $\operatorname{Prob}^{\mu,\, \varepsilon}$ refer to the law for the solution of \eqref{SDE1} starting from an initial measure $\mu$, which we assume for simplicity to have a smooth, positive, rapidly decaying density with respect to the Lebesgue measure on $\mathbb{R}^d$. The large deviation rate function $I^\varepsilon: \mathbb{R} \to [0,\infty]$ in this problem is the function that gives the exponential rate of decay in $t$ of fluctuations of order $t$ in $S^\varepsilon_t$,
\begin{equation}
\label{eq:rough-LDP}
    \operatorname{Prob}^{\mu,\varepsilon} \{ t^{-1}S^\varepsilon_t \approx s\} \asymp \exp\left(-t I^\varepsilon(s)\right)
\end{equation}
as $t\to\infty$; see Section \ref{sec: Feynman--Kac semigroups} for a more precise formulation of the large deviation principle. We are interested in an efficient way of numerically computing this rate function.

Before we discuss numerical considerations, let us briefly explain how the rate function is related to an eigenvalue computation.  The moment-generating function of $S_t^\varepsilon$ with respect to $\operatorname{Prob}^{\mu,\, \varepsilon}$ is
\begin{equation}
    \chi_{t}^{\varepsilon}(\alpha)
    = \int_{C_t}\, \exp( -\alpha S_{t}^{\varepsilon} ) \, \d\!\operatorname{Prob}^{\mu,\, \varepsilon},
\label{MGF1}
\end{equation}
where $\alpha \in \bR$ and $C_t$ is the space $C([0, t]; \mathbb{R}^d)$ of continuous paths in $\mathbb{R}^d$ over the time interval $[0, t]$.
Under our assumptions, the following Feynman--Kac representation of the moment-generating function $\chi_{t}^{\varepsilon}(\alpha)$ holds:
\begin{equation}
    \chi_{t}^{\varepsilon}(\alpha)
    = \int_{\mathbb{R}^d}  \left(\exp( t \,  \sA^{\varepsilon, \alpha} ) \1 \right)(\xi) \, \d \mu(\xi)\,,
\label{chi1}
\end{equation}
where the operator $\sA^{\varepsilon, \alpha}$ is a second-order differential operator that is elliptic but not self-adjoint.
Such a representation dates at least back to \cite{Ku98,LS99} and relies on Girsanov's theorem and the Feynman--Kac formula; we refer to \cite{BDG15,raquepas2020large} for rigorous proofs that cover our hypotheses. With $\lambda^{\varepsilon, \alpha}$ the principal eigenvalue (the one with the largest real part) of $\sA^{\varepsilon, \alpha}$, the identity
\begin{equation}
    \lim_{t \rightarrow \infty} \frac{1}{t}
    \log \chi_{t}^{\varepsilon}(\alpha) = \lambda^{\varepsilon, \alpha} \label{chi2}
\end{equation}
provides a spectral-theoretic point of view on the large-$t$ behavior of the moment-generating function, which is instrumental in the study of large deviations. The moment-generating function is convex in $\alpha$ and symmetric about $\alpha=\tfrac 12$. The spectral-theoretic point of view provides tools for showing smoothness in $\alpha$. The Legendre transform of $\lambda^{\varepsilon, \alpha}$ in the variable $\alpha$ is the large deviation rate function $I^\varepsilon$ in \eqref{eq:rough-LDP}:
\begin{equation} \label{eq: Legendre transform}
I^\varepsilon(s) = \sup_{\alpha} \left(-\alpha s - \lambda^{\varepsilon, \alpha}\right).
\end{equation}
The symmetry about $\alpha = \tfrac 12$ gives rise to the Gallavotti--Cohen symmetry $I^\varepsilon(-s) = I^\varepsilon(s) + s$. In sufficiently regular situations, many statistical properties of the family $(S_t^{\varepsilon})_{t>0}$ can be equivalently read off the limiting cumulant-generating function $\lambda^{\varepsilon,\alpha}$ or off the rate function $I^{\varepsilon}(s)$. For example, the asymptotic mean entropy production per unit time is both $-\partial_\alpha\lambda^{\varepsilon,\alpha}|_{\alpha = 0}$ and the zero of $I^\varepsilon$.
Again, we refer to
\cite{BDG15,JPS17,raquepas2020large} for proofs and more thorough theoretical discussions.

There are several motivations for seeking novel numerical methods for accessing $I^\varepsilon$ via $\lambda^{\varepsilon, \alpha}$. First, trying to probe the large deviations of $S_t^\varepsilon$ from direct simulations of \eqref{SDE1} and computation of \eqref{eq:def-Seps} is not realistic since these large deviations are events with exponentially small probabilities. In most cases where rigorous theorems on entropy production are proved, $\lambda^{\varepsilon, \alpha}$ is the only available access to the rate function $I^\varepsilon$, but admits no closed-form formula.
Second, the assumptions for these theorems are relatively stringent\,---\,most significantly by the non-degeneracy assumption on the noise\,---\,and we are looking for ways to explore the large deviations in situations where no rigorous results are available.
We will be particularly interested in the small-noise regime $0 < \varepsilon \ll 1$ since, under additional assumptions at the critical points of $V$, \cite{raquepas2020large} provides explicit formulas for the limits $\lambda^{0,\alpha} = \lim_{\varepsilon \to 0^+}\lambda^{\varepsilon,\alpha}$ for $\alpha$ in an interval of the form $(-\delta,1+\delta)$ and $I^{0}(s) = \lim_{\varepsilon \to 0^+}I^{\varepsilon}(s)$ for $s$ in an interval around the mean entropy production rate, allowing us to compare our numerical results.\footnote{The physical and technical reasons for the restriction to values of $\alpha$ near the interval $[0,1]$ are beyond the scope of the present article; we refer the reader to \cite{JPS17,raquepas2020large,vZC03}. All of our numerical experiments abide by the appropriate restrictions on $\alpha$, except for Example \ref{example: 2D_circle}.} In the presence of momentum variables, the vanishing-noise limit has attracted independent interest in the physics literature since \cite{Ku07}, due to its relation to deterministic systems; it still does to this day \cite{BGL22,raquepas2020large,Mo23}. We will come back to this point in Section \ref{sec: Feynman--Kac semigroups}.

In this paper, we develop an interacting particle method (IPM) \cite{del2004feynman,doucet2001sequential,ferre2019error,hairer2014improved,lelievre2010free} for numerically computing $\lambda^{\varepsilon, \alpha}$\,---\,and thus $I^\varepsilon(s)$\,---\,at $0 < \varepsilon \ll 1$.
    More precisely, we consider an $\alpha$- and $\varepsilon$-dependent, discrete-time semigroup obtained from an operator splitting and an Euler--Maruyama scheme with a time step size $\Delta t$, and then show that the spectral radius associated with this discrete-time semigroup has the following two properties:
    \begin{itemize}
        \item on the one hand, it is accessible through large iterates of the semigroup and lends itself to the IPM, thanks to suitable stability properties \cite{ferre2021more}; see Propositions \ref{prop: continuous stability} and \ref{prop: discrete stability};
        \item on the other hand, it provides a good approximation of $\lambda^{\varepsilon, \alpha}$ for small $\Delta t$, thanks to different results from (non-self-adjoint) perturbation theory \cite{AP68,Kat,Tr59}; see Theorem \ref{thm: convergence of IPM wrt time step size}, which appears to be the first convergence result of time discretization of Feynman--Kac semigroups in unbounded domains for computing the principal eigenvalue.
    \end{itemize}
We also discuss techniques for setting the measure of initial conditions to obtain faster approximations of this spectral radius.

To put things into perspective, let us briefly discuss the computational difficulties. A conventional way to obtain $\lambda^{\varepsilon, \alpha}$ numerically is to compute it directly from $\sA^{\varepsilon, \alpha}$, by first discretizing $\sA^{\varepsilon, \alpha}$ using traditional mesh-based methods, e.g., finite element methods \cite{sun2016finite} or finite difference methods \cite{carasso1969finite,kuttler1970finite}, and then solving the principal eigenvalue problem of the resulting non-symmetric matrix using, e.g., QR methods \cite{watkins1993some} or Arnoldi methods \cite{saad1980variations,sorensen1992implicit}. However, the following three issues pose great challenges to such kinds of methods.
\begin{enumerate}
    \item[1.] \emph{Unboundedness of the physical domain}: Since the stochastic dynamics \eqref{SDE1} is defined in all of $\mathbb{R}^d$, truncation of the domain is usually needed in mesh-based methods \cite{han2013artificial}, and this may introduce numerical errors.
    \item[2.] \emph{High dimensionality}: Having in mind applications to stochastic thermodynamics in which the dimension $d$ of $X_t$ in \eqref{SDE1} is proportional to the number of particles, we would like to be able to handle situations where $d$ is large, but most mesh-based methods suffer from the curse of dimensionality.
    \item[3.] \emph{Singularities in the vanishing-noise limit}: With $\psi^{\varepsilon, \alpha}$ the normalized principal eigenfunction, it is known from \cite{fleming1997asymptotics} that $\varepsilon \log \psi^{\varepsilon, \alpha}$ has a nontrivial limit as $\varepsilon \rightarrow 0^+$ under certain additional conditions. This implies that $\psi^{\varepsilon, \alpha}$ is asymptotically proportional to $\exp(- \varepsilon^{-1} \Psi^\alpha)$ for some function $\Psi^\alpha$ and thus admits singularities in the vanishing-noise limit. For mesh-based methods, finer grids are needed to capture the singularity.
\end{enumerate}
On the other hand, the IPM provides an alternative to the computation of $\lambda^{\varepsilon, \alpha}$ from the perspective of Feynman--Kac semigroups, which has already been applied to the computation of ground state energies of Schr\"{o}dinger operators using diffusion Monte Carlo \cite{anderson1975random,ceperley1980ground,foulkes2001quantum,grimm1971monte}, to the computation of effective diffusivity \cite{lyu2020convergence,wang2018computing,wang2018sharp,wang2022computing} and KPP front speeds \cite{lyu2022convergent,zhang2023convergent}, and to non-linear filtering problems \cite{del1997nonlinear,del1999central,del2000branching}, to mention only a few. Since the IPM is based on simulations of an SDE, it naturally applies to unbounded domains and it is independent of whether the operator whose principal eigenvalue is sought is self-adjoint. In addition, the IPM is an analogue of diffusion Monte Carlo, which uses two randomization steps to avoid exponential explosion of complexity in $d$ for a fixed accuracy level, i.e., approximating Feynman--Kac semigroups by distributions and reducing variance by resampling \cite{lim2017fast}. Also, the IPM scales easily to high dimensions in terms of coding.
In numerical examples of different values of $d$ (up to 16), the numerical approximation of $\lambda^{\varepsilon, \alpha}$ converges to its predicted vanishing-noise limit within visual tolerance with a fixed number of particles and a fixed time step size, which shows the scalability and robustness of our method for large $d$ and small $\varepsilon$. Moreover, the empirical density of the particles we obtain at the final time (after resampling) accurately captures singularities of the vanishing-noise limit, which is compatible with \cite{fleming1997asymptotics}. We point out here that Feynman--Kac semigroups have a long history in large deviation theory. 
Theoretical studies focused on establishing large deviation principles in different applications with variational formulae developed for rate functions that can be accessed through Feynman--Kac semigroups; see, e.g., \cite{hollander2000large,donsker1975variational,kontoyiannis2005large,touchette2009large,varadhan1984large,wu2001large}.
Numerically, the population dynamics and its variant,  which are analogues of IPM, were applied to the calculation of rate functions; see, e.g., \cite{giardina2006direct,hidalgo2017finite,lecomte2007numerical,nemoto2016population,nemoto2017finite,tailleur2009simulation}. However, few theoretical results on the properties of these methods are found; one exception is \cite{nemoto2017finite}, in which systematic errors and stochastic errors of the population dynamics were studied but only in the setting of a discrete and finite state space (i.e., the physical domain in our context). Moreover, although there are numerical examples of computing rate functions using the population dynamics for high-dimensional cases (see, e.g., \cite{giardina2006direct,lecomte2007numerical,tailleur2009simulation}), these computations were still limited to the setting of a discrete and finite state space. To the best of our knowledge, in the setting of a non-compact physical domain there are so far no numerical results of the challenging high-dimensional case using analogues of IPM.

The rest of this paper is organized as follows. In Section \ref{sec: Feynman--Kac semigroups}, we present the Feynman--Kac semigroup formulation of the principal eigenvalue problem and the formulation of the large deviation principle. In Section \ref{sec: numerical discretization}, we introduce the discrete-time semigroup at the heart of our numerical approximation and present our theoretical results on the corresponding spectral radius.
In Section \ref{sec: IPM}, we present the interacting particle algorithm and techniques for setting the initial measure.
We begin Section \ref{sec: numerical examples} with numerical examples of dimensions up to 16, exploring the vanishing-noise limit. We find excellent agreement of our numerical examples with the explicit theoretical predictions when the analytical vanishing-noise limits exist in tractable form; another numerical experiment allows us to probe situations for which we are not aware of explicit theoretical predictions for the vanishing-noise limit. We end Section \ref{sec: numerical examples} with convergence tests of our IPM. Finally, we give some concluding remarks in Section \ref{sec: conclusion}. 
We provide proof sketches of propositions and theorems in Appendix \ref{append: proofs}.

\paragraph{Notation} Let $\cP(\bR^d)$ be the space of all probability measures over $\bR^d$. For a measure $\mu$ with finite mass, let $\pair{\mu}{\varphi} = \int_{\bR^d} \varphi \,\d\mu$ for any $\varphi \in L^{\infty}(\bR^d)$.
We use $C_0$ for the space of continuous real-valued functions on $\bR^d$ that vanish at infinity, and
given a function $W: \bR^d \rightarrow [1, +\infty)$, we use the notation
\begin{align}
    L^{\infty}_W(\bR^d) = \left\{ \varphi\in L^\infty_\textnormal{loc}(\bR^d) :  \left\| \frac{\varphi}{W} \right\|_{L^{\infty}(\bR^d)} < +\infty \right\}.
\end{align}
We use $|\,\cdot\,|$ for the Euclidean norm on $\bR^d$ and $\|\,\cdot\,\|$ for the operator norm it induces on $d$-by-$d$ matrices.

\section{Continuous-time Feynman--Kac semigroups and large deviations}
\label{sec: Feynman--Kac semigroups}

This section, together with Section \ref{sec: numerical discretization}, serves to show that the principal eigenvalue $\lambda^{\varepsilon, \alpha}$ of $\sA^{\varepsilon, \alpha}$ in \eqref{MGF1}--\eqref{chi2}
can be approximated using the spectral radius of a discrete-time semigroup, which will then be combined with particle systems and resampling to yield our IPM in Section \ref{sec: IPM}. Before proceeding further, we make some assumptions on $V$ and $b$, trying to strike a balance between optimality and readability.

\begin{assumption}
\label{assumption: quadratic growth of potential}
    We assume that (1) $V \in C^{\infty}(\mathbb{R}^d)$; (2) there exists a positive-definite matrix $H_0$ such that $\langle x, H_0 \nabla V(x) \rangle \geq |x|^2$ whenever $|x|$ is large enough; (3) $\|D^2 V(x)\| = o(|\nabla V(x)|)$ as $|x| \to \infty$.
\end{assumption}

\begin{assumption}
\label{assumption: bounded velocity}
    We assume that (1) $b \in (C^{\infty}(\mathbb{R}^d))^d$; (2) $\| b \|_{C^1(\mathbb{R}^d)} < \infty$; (3) $\langle b,\nabla V\rangle \leq c|\nabla V|^2$ for some constant $0 < c < \tfrac 12$.
\end{assumption}

Recall that the formulas \eqref{chi1}--\eqref{chi2} appealed to an elliptic operator $\mathscr{A}^{\varepsilon, \alpha}$ as the generator of a semigroup. On a suitable function space, this semigroup is compact and irreducible and this is what guarantees that the principal eigenvalue $\lambda^{\varepsilon, \alpha}$ of $\mathscr{A}^{\varepsilon, \alpha}$ appropriately captures the large-$t$ behavior of the moment-generating function $\chi_t^{\varepsilon}(\alpha)$; see e.g. \cite{BDG15,raquepas2020large}.
For the analysis of the present paper, we will instead work with the spectrally equivalent operator
\begin{align}
    \cA^{\varepsilon, \alpha} f := &
    \exp((- 2 \varepsilon)^{-1} V) \sA^{\varepsilon, \alpha} (\exp((2 \varepsilon)^{-1} V)) \nonumber \\
    \phantom{:}= &
    \varepsilon \Delta f + \langle (1 - 2 \alpha) b, \nabla f \rangle - \frac{1}{4 \varepsilon} |\nabla V|^2 f + \frac{1}{2 \varepsilon} \langle b, \nabla V \rangle f - \frac{\alpha (1 - \alpha)}{\varepsilon} |b|^2 f + \frac{1}{2} (\Delta V) f - \alpha (\nabla \cdot b) f,
\end{align}
which is also associated with a semigroup, which we now take the time to describe.
Define the operator $\mathcal{L}^{\varepsilon, \alpha}$ by
\begin{align}
\mathcal{L}^{\varepsilon, \alpha} f = \varepsilon \Delta f + \langle(1 - 2 \alpha) b, \nabla f\rangle
\end{align}
on sufficiently regular functions and let
\begin{align}
\label{eq:def-U}
U^{\varepsilon, \alpha} = & - \frac{1}{4 \varepsilon} |\nabla V|^2 + \frac{1}{2 \varepsilon} \langle b, \nabla V \rangle - \frac{\alpha (1 - \alpha)}{\varepsilon} |b|^2 + \frac{1}{2} (\Delta V) - \alpha (\nabla \cdot b).
\end{align}

For readability, let us fix $\alpha$ and $\varepsilon$ and omit the dependence on $\alpha$ and $\varepsilon$ from the notation for the time being.
Consider the SDE with infinitesimal generator $\mathcal{L}$, i.e.,
\begin{align} \label{eq: continuous SDE}
\d X_t = (1 - 2 \alpha) b \,\d t + \sqrt{2 \varepsilon} \,\d B_t,
\end{align}
with $B_t$ a $d$-dimensional Brownian motion, and the evolution operator $P^U_t$ defined by
\begin{align}
    P^U_t \varphi(x) = \E \left[\varphi(X_t) \exp \left(\int_0^t U(X_s) \d s \right) \Big{|} X_0 = x \right],
\end{align}
where $\E$ is the expectation over all realizations of \eqref{eq: continuous SDE} and $\varphi$ is a function in a suitable space.
With natural choices of domain and space, $\cA = \cL + U$ is indeed the generator of the positivity-preserving semigroup $(P^U_t)_{t > 0}$ with the same desirable properties as that generated by $\mathscr{A}$\,---\,albeit on a different space. While these properties and their consequences can be obtained in many different ways, we present a result that foreshadows our upcoming analysis of the discrete semigroups behind our IPM, and the result is an application of \cite[Theorem 3]{ferre2021more} to our problem; see Appendix \ref{append: proofs} for a proof sketch.

\begin{proposition} \label{prop: continuous stability}
Let $W(x) = \mathrm{e}^{\theta|x|^2}$ and suppose that Assumptions \ref{assumption: quadratic growth of potential}--\ref{assumption: bounded velocity} hold. For $\theta>0$ small enough, there exists a unique measure $\mu^{\star}_U \in \cP(\bR^d)$ with $\pair{\mu^{\star}_U}{W} < +\infty$ and a constant $\kappa > 0$ with the following property: for any initial measure $\mu \in \cP(\bR^d)$ with $\pair{\mu}{W} < + \infty$, there exists a constant $C_{\mu} > 0$ such that
\begin{align}
\label{eq: convergence of conitinous Feynman--Kac semigroup}
\left| \frac{\pair{\mu}{P^U_t \varphi}}{\pair{\mu}{P^U_t \1}} - \pair{\mu^{\star}_U}{\varphi}\right| \le C_{\mu} \mathrm{e}^{- \kappa t} \| \varphi \|_{L^{\infty}_W}
\end{align}
for all $\varphi \in L^{\infty}_W(\bR^d)$ and $t > 0$. Moreover, the leading eigenvalue $\lambda$ of $\mathscr{A}$ is simple and real, and
\begin{align}
\label{eq: computation of eigenvalue in contiuous case}
\lambda = \lim \limits_{t \rightarrow \infty} \frac{1}{t} \log \E \left[ \exp \left( \int_0^t U(X_s) \d s \right) \bigg{|} X_0 \sim \mu \right],
\end{align}
where the expectation runs over the initial condition $X_0 \sim \mu$ and all realizations of \eqref{eq: continuous SDE}.
\end{proposition}

We now reintroduce the dependence on $\alpha$ and $\varepsilon$ in the notation.
We also note that it follows from standard perturbation-theory arguments that the limiting function $\alpha \mapsto \lambda^{\varepsilon, \alpha}$ is real-analytic. Hence, by \eqref{chi2} and the G\"artner--Ellis theorem, the following large deviation principle holds: with $I^\varepsilon$ the Legendre transform of the function $\alpha \mapsto \lambda^{\varepsilon, \alpha}$, we have
\begin{align*}
    -\inf_{s \in \operatorname{int} E} I^\varepsilon(s)
        \leq \liminf_{t\to\infty} \frac 1t \log \operatorname{Prob}^{\mu,\varepsilon}\left\{\tfrac 1t S_t^\varepsilon \in E \right\} \leq \limsup_{t\to\infty} \frac 1t \log \operatorname{Prob}^{\mu,\varepsilon}\left\{\tfrac 1t S_t^\varepsilon \in E \right\}
        \leq -\inf_{s \in \operatorname{cl} E} I^\varepsilon(s)
\end{align*}
for every Borel set $E\subseteq \mathbb{R}$; again see \cite{BDG15,raquepas2020large}. It was shown in \cite{raquepas2020large} that, locally and under additional conditions at the critical points of $V$, easily accessible formulas can be given in the subsequent limit $\varepsilon \to 0^+$, without any rescaling of $\lambda^{\varepsilon, \alpha}$ or $I^\varepsilon$. Roughly speaking, this means that we get easy access to a limiting rate function $I^0$ such that
\begin{equation}
\label{eq:rough-LDP-0}
        \operatorname{Prob}^{\mu,\varepsilon}\{ t^{-1}S^\varepsilon_t \approx s\} \asymp \exp\left(-t I^0(s)\right)
\end{equation}
for $t \gg \varepsilon^{-1} \gg 1$ and $s$ near the mean entropy production rate.
These additional conditions will be met here if we further assume that
\begin{equation}
\label{eq:ND-on-V}
    \det D^2 V|_{x_j} \neq 0
\end{equation}
at each of the finitely many critical points $\{x_j\}_{j=1}^J$ of $V$, and that
\begin{equation}
\label{eq:ND-on-b-vs-V}
    |b(x_j + \xi)| = O( |D^2V|_{x_j} \xi|).
\end{equation}
These extra conditions force the deterministic dynamics obtained by plainly putting $\varepsilon = 0$ in \eqref{SDE1} to have only very simple invariant structures. The limiting $\lambda^{0,\alpha}$ turns out to be the principal eigenvalue for a quadratic approximation of $\cA$ at some  $\alpha$-dependent choice of critical point of $V$, in such a way that the limiting $I^{0}$ is the convex envelope of different rate functions that would arise from linear diffusions approximating \eqref{SDE1} near critical points of $V$.

Suppose on the contrary, that \eqref{eq:ND-on-V}--\eqref{eq:ND-on-b-vs-V} fail, say because $V$ has a whole critical circle to which $b$ is tangent as in \cite[Section 5]{BDG15}. Then, we expect to see, as $\varepsilon \to 0^+$, the principal eigenvalue $\lambda^{\varepsilon, \alpha}$ diverge for $\alpha \notin [0,1]$.
In such situations, one can consider the rescaling of \cite{BDG15,BGL22} to obtain further information on the behaviour of those divergences and their relations to the deterministic dynamics and Freidlin--Wentzell theory. We will explore this numerically in Section \ref{sec: numerical examples}.

\section{Time discretization using discrete-time semigroups}
\label{sec: numerical discretization}

We again fix $\alpha$ and $\varepsilon$ and omit keeping track of them in the notation. To compute the principal eigenvalue $\lambda$, we consider a time discretization of the operator semigroup $(P^U_t)_{t > 0}$, which consists of two steps: an operator splitting scheme and an Euler--Maruyama scheme for the SDE \eqref{eq: continuous SDE}.

With a time step size $\Delta t > 0$, define an evolution operator $\wP^U_{\Delta t}$ by
\begin{align}
\wP^U_{\Delta t} \varphi(x) = \exp(\Delta t U(x)) \E \left[ \varphi(X_{\Delta t}) | X_0 = x \right],
\end{align}
where $X_{\Delta t}$ is the solution to \eqref{eq: continuous SDE} at time $\Delta t$ and $\varphi$ is a function in a suitable space. Note that if we define an operator $P_t$ by
\begin{align} \label{eq: diffusion operator}
P_t \varphi(x) = \E \left[ \varphi(X_t) | X_0 = x \right],
\end{align}
then $P_t = \exp(t \cL)$ on a suitable space. Hence, $\wP^U_{\Delta t} = \exp(\Delta t U) \exp(\Delta t \cL)$ can be seen as an approximation of $P^U_{\Delta t}$ using an operator splitting scheme. One can show using arguments similar to those in the proof of Proposition \ref{prop: continuous stability} for $P^U_{\Delta t}$ (Appendix \ref{append: proofs}) that the spectral radius $\wLambda_{\Delta t}$ of $\wP^U_{\Delta t}$ admits a positive eigenvector and that no other eigenvalue admits a positive eigenvector.
It should be expected that, for $\Delta t \ll 1$, we have
$\log \wLambda_{\Delta t} \approx {\Delta t} \lambda$.
We will come back to this point at the end of this section.

We now further discretize $\wP^U_{\Delta t}$ by considering an Euler--Maruyama scheme for \eqref{eq: continuous SDE} with the time step size $\Delta t$, which reads
\begin{align}
    \label{eq: Euler scheme}
    \begin{cases}
        \widehat{X}_{n + 1} = \widehat{X}_n + (1 - 2 \alpha) b(\widehat{X}_n) \Delta t + \sqrt{2 \varepsilon \Delta t} G_n, \\
        \quad\, \widehat{X}_0 \sim \mu,
    \end{cases}
\end{align}
where $G_n$ is a $d$-dimensional standard Gaussian random variable. We define the evolution operator $\hP_{\Delta t}$ by
\begin{align}
    \hP_{\Delta t}\varphi(x) = \E[\varphi(\widehat{X}_{n + 1}) | \widehat{X}_n = x],
\end{align}
and define $\hP_{\Delta t}^U$ by
\begin{align}
    \hP_{\Delta t}^U\varphi(x) = \exp(U(x) \Delta t) \hP_{\Delta t}\varphi(x).
\end{align}
In view of the good convergence properties of the Euler--Mayurama scheme and the growth of $U$,
we expect the spectral radius $\hLambda_{\Delta t}$ of $\hP_{\Delta t}^U$ to satisfy
$
    \log \hLambda_{\Delta t} \approx
    \log \wLambda_{\Delta t}
$
for $\Delta t \ll 1$.

We define a normalized, discrete-time, dual Feynman--Kac semigroup
associated with $\hP_{\Delta t}^U$ by
\begin{align}
\label{eq: FK semigroup of Euler scheme}
\pair{\Phi_{k, \Delta t}\mu}{\varphi} = \frac{\pair{\mu}{(\hP_{\Delta t}^U)^k \varphi}}{\pair{\mu}{(\hP_{\Delta t}^U)^k \1}} = \frac{\E \left[ \varphi(\widehat{X}_k) \exp \left( \Delta t \sum_{j = 0}^{k - 1} U(\widehat{X}_j) \right) \Big{|} \widehat{X}_0 \sim \mu \right]}{\E \left[ \exp \left( \Delta t \sum_{j = 0}^{k - 1} U(\widehat{X}_j) \right) \Big{|} \widehat{X}_0 \sim \mu \right]}
\end{align}
for any initial measure $\mu$ and any bounded measurable function $\varphi$. 
The following proposition, which is an application of \cite[Theorem 1]{ferre2021more} to our problem, establishes desirable stability properties of $\Phi_{k,\Delta t}$ for the purpose of numerically accessing the spectral radius $\hLambda_{\Delta t}$; see Appendix \ref{append: proofs} for a proof sketch.

\begin{proposition} \label{prop: discrete stability}
    Suppose that Assumptions \ref{assumption: quadratic growth of potential}--\ref{assumption: bounded velocity} hold. Then, there exists a measure $\hmu^{\star}_{U, \Delta t} \in \cP(\bR^d)$ with $\hLambda_{\Delta t} = \pair{\hmu_{U,\Delta t}^\star}{\hP_{\Delta t}^U\1}$  and a constant $\widehat{\beta} \in (0, 1)$ with the following property: for any initial measure $\mu \in \cP(\bR^d)$, there is a constant $C_{\mu}$ for which
    \begin{align}
    \label{eq: stability of Euler scheme}
        | \pair{\Phi_{k, \Delta t}\mu}{\varphi} - \pair{\hmu^{\star}_{U, \Delta t}}{\varphi}| \le C_{\mu} \widehat{\beta}^k || \varphi ||_{L^{\infty}}
    \end{align}
    for all $\varphi \in L^{\infty}(\bR^d)$ and $k \ge 1$. Moreover,
    \begin{align}
    \label{eq: CGF in Euler scheme}
        \log \hLambda_{\Delta t} = \lim \limits_{k \rightarrow \infty} \frac{1}{k} \log \E \Bigg[ \exp \Bigg( \Delta t \sum_{j = 0}^{k - 1} U(\widehat{X}_j) \Bigg) \Bigg{|} \widehat{X}_0 \sim \mu \Bigg].
    \end{align}
\end{proposition}

The aforementioned intuition that
$
    \log \hLambda_{\Delta t} \approx \log \wLambda_{\Delta t} \approx {\Delta t}\lambda
$
for $\Delta t \ll 1$
can indeed be turned into the following soft convergence result; see Appendix \ref{append: proofs} for a proof sketch.

\begin{theorem} \label{thm: convergence of IPM wrt time step size}
    Under Assumptions \ref{assumption: quadratic growth of potential}--\ref{assumption: bounded velocity} we have
    \[
        \lim_{n \to \infty} n \log \hLambda_{Tn^{-1}} = \lim_{n\to\infty} n \log \wLambda_{Tn^{-1}} = T\lambda.
    \]
    for every $T>0$.
\end{theorem}

\begin{remark}
Let $\hlambda_{\Delta t} = \frac{\log \hLambda_{\Delta t}}{\Delta t}$. Theorem \ref{thm: convergence of IPM wrt time step size} shows that $\lim \limits_{\Delta t \rightarrow 0} \hlambda_{\Delta t} = \lambda$, but the convergence order is unknown. In fact, in the case of compact physical domains, the convergence analysis in \cite{ferre2019error} shows that $| \hlambda_{\Delta t} - \lambda | = \mathcal{O}(\Delta t)$ as $\Delta t \rightarrow 0$, the proof of which relies on the property that $V$ and $b$ are bounded in $C^2$ and $C^1$, respectively, over the physical domain. However, $V$ is unbounded in our case, and the generalization of the aforementioned convergence analysis to non-compact physical domains is non-trivial as pointed out in \cite{ferre2019error}. We will leave it to a future study.
\end{remark}

\section{Interacting particle methods}
\label{sec: IPM}

Sections \ref{sec: Feynman--Kac semigroups} and \ref{sec: numerical discretization} show that the principal eigenvalue $\lambda^{\varepsilon, \alpha}$ can be approximated in terms of the logarithmic spectral radius in \eqref{eq: CGF in Euler scheme} that is accessible through large iterates of a discrete-time semigroup with good stability properties. We obtain the full numerical discretization by using a particle system to approximate the expectation in \eqref{eq: CGF in Euler scheme}, i.e., the IPM. Given an ensemble of particles, the IPM proceeds within each time interval as follows. The particles evolve according to the dynamics of $\hP_{\Delta t}$ with an importance weight assigned to each particle, and then to control variance \cite{ferre2019error} (or to avoid weight degeneracy \cite{lelievre2010free}) the particles are resampled according to the multinomial distribution associated with their respective weights. The logarithmic spectral radius in \eqref{eq: CGF in Euler scheme} is accessed using the particles at each time step. The complete algorithm of the IPM for computing $\lambda^{\varepsilon, \alpha}$ is given in Algorithm \ref{algo: IPM}, where we only emphasize the dependence of the final approximation $\hlambda^{\varepsilon, \alpha}_{\Delta t}$ on $\varepsilon$ and $\alpha$. Note that the particles $\{ \mathbf{q}^{n, m} \}_{m = 1}^M$ are no longer independent as soon as $n\geq 1$ but still exchangeable.

\begin{algorithm}[ht]
\caption{The interacting particle method for computing $\lambda^{\varepsilon, \alpha}$ \label{algo: IPM}}
\begin{algorithmic}
\STATE{\textbf{Input}: $\alpha$, noise level $\varepsilon$, velocity field $b$, potential $V$, number of particles $M$, initial measure $\mu$, final time $T$, time step size $\Delta t = \frac{T}{N}$.}
\STATE{Generate $M$ independent and $\mu$-distributed particles $\{\mathbf{q}^{0, m}\}_{m = 1}^M$.}
\FOR{n = 1:N}
\STATE{Compute each $\widetilde{\mathbf{q}}^{n, m}$ using the Euler--Maruyama scheme \eqref{eq: Euler scheme} with $\mathbf{q}^{n - 1, m}$ the initial value.}
\STATE{Compute each weight $w^{n - 1, m} = \exp(\Delta t U(\mathbf{q}^{n - 1, m}))$ according to \eqref{eq:def-U}.}
\STATE{Compute the quantities $P^{n - 1} = \sum_{m = 1}^M w^{n - 1, m}$ and $\hlambda^{n - 1} = \log(P^{n - 1} / M)$.}
\STATE{Compute the probabilities $p^{n - 1, m} = w^{n - 1, m} / P^{n - 1}$
and
sample $M$ non-negative integers $(K_m)_{m=1}^M$ summing to $M$ according to the multinomial law
\begin{align*}
\operatorname{Prob}\{K_1 = k_1, \ldots, K_M = k_M\} = \frac{M!}{\prod_{m = 1}^M k_m!} \prod_{m = 1}^M (p^{n - 1, m})^{k_m}.
\end{align*}}
\STATE{Set $(\mathbf{q}^{n, m})_{m = 1}^M$ to contain $K_m$ copies of $\widetilde{\mathbf{q}}^{n, m}$.}
\ENDFOR
\STATE{Compute the approximation
\begin{align} \label{eq: numerical approximation of principal eigenvalue}
    \hlambda^{\varepsilon, \alpha}_{\Delta t} = \frac{1}{T} \sum_{n = 0}^{N - 1} \hlambda^{n}
\end{align}
of the principal eigenvalue.}
\STATE{\textbf{Output}: the approximation $\hlambda^{\varepsilon, \alpha}_{\Delta t}$ of the principal eigenvalue.}
\end{algorithmic}
\end{algorithm}

\subsection{The empirical measure of particles at the final time}
\label{ssec:empirical-meas}

The empirical measure of particles at the final time $T$ (equivalently after the $N$-th step) is a random measure that is thought of as an approximation to the Feynman--Kac semigroup $\Phi_{N, \Delta t}\mu$, defined in \eqref{eq: FK semigroup of Euler scheme}, and thus to the invariant measure $\hmu^{\star}_{U, \Delta t}$ for $N$ large by \eqref{eq: stability of Euler scheme}.
To justify this, let us first consider the one-step evolution $\Phi_{1, \Delta t}\mu$.
First of all, by the Glivenko--Cantelli theorem or a variant thereof (see e.g. \cite{Fortet1953convergence,talagrand1987glivenko}), the empirical measure
$$
    \hmu_{U,\Delta t,M}^{0,+}
    :=
    \frac{1}{M}
    \sum_{m = 1}^M \delta_{\mathbf{q}^{0, m}}
$$
of the particles $\{\mathbf{q}^{0, m}\}_{m = 1}^M$ approximates $\mu$ well provided that $M$ is large.
Then, on the one hand, by \eqref{eq: FK semigroup of Euler scheme} and the definitions in Algorithm \ref{algo: IPM}, the measure $\Phi_{1, \Delta t}\mu$ can be approximated by the weighted empirical measure
\begin{align*}
\hmu^{1,-}_{U, \Delta t, M} := \Phi_{1,\Delta t} \hmu_{U,\Delta t,M}^0 = \sum_{m = 1}^M p^{0, m} \delta_{\widetilde{\mathbf{q}}^{1, m}}.
\end{align*}
On the other hand, since the multinomial law used in Algorithm \ref{algo: IPM} satisfies
\[
    \sum_{k_1, \dotsc, k_M} k_m  \operatorname{Prob}\{K_1 = k_1, \ldots, K_M = k_M\} = M p^{0,m},
\]
we have that, for any test function $\varphi$,
\begin{align*}
    & \sum_{k_1, \dotsc, k_M} \left(\frac 1M \sum_{m=1}^M \varphi(\mathbf{q}^{1, m})\right) \operatorname{Prob}\{K_1 = k_1, \ldots, K_M = k_M\}
    \\ & \qquad \qquad
        = \sum_{k_1, \dotsc, k_M} \left( \sum_{m=1}^M \frac{k_m}{M} \varphi(\tilde{\mathbf{q}}^{1, m})\right)  \operatorname{Prob}\{K_1 = k_1, \ldots, K_M = k_M\}
    \\ & \qquad \qquad
        =  \sum_{m=1}^M  \sum_{k_1, \dotsc, k_M} \frac{k_m}{M}  \operatorname{Prob}\{K_1 = k_1, \ldots, K_M = k_M\} \varphi(\tilde{\mathbf{q}}^{1, m})
    \\ & \qquad \qquad
        = \sum_{m=1}^M p^{0,m} \varphi(\tilde{\mathbf{q}}^{1, m}).
\end{align*}
Hence, the resampled empirical measure $\hmu_{U,\Delta t, M}^{1,+}$ of $\{ \mathbf{q}^{1, m} \}_{m = 1}^M$ yields, once the randomness in the resampling process is averaged out, the same expectations as the weighted empirical measure $\hmu^{1,-}_{U, \Delta t, M}$. In particular, this holds when we let $\varphi = \hP_{\Delta t}^U \exp(\Delta t U)$, which is relevant at the next step for carrying on with our approximation of the principal eigenvalue. It is expected that, for that purpose and when $M$ is large, the empirical measure $\hmu_{U,\Delta t, M}^{1,+}$ is a numerically sounder choice as it gives more importance to the regions where $\exp(\Delta t U)$ is large. We refer the readers to, e.g., \cite{del2004feynman} and \cite{lelievre2010free}, for more thorough discussions.

Iterating this argument,
the measure $\Phi_{N, \Delta t}\mu$ should indeed be well approximated by the resampled empirical measure $\hmu_{U,\Delta t, M}^{N,+}$ of $\{ \mathbf{q}^{N, m} \}_{m = 1}^M$.
In our numerical examples in Section \ref{sec: numerical examples}, as $\varepsilon \rightarrow 0^+$, the observed asymptotic behavior of the empirical density of particles at $T$ is consistent with the theory in \cite{fleming1997asymptotics}.

\subsection{Choice of the initial measure}

The IPM involves the choice of an initial measure $\mu$ for the particles.
The effect of this choice is the strongest on terms in the sum \eqref{eq: numerical approximation of principal eigenvalue} for which $n\Delta t \ll 1$. For example, the term with $n=1$ approximately contributes $\log\pair{\Phi^U_{\Delta t}\mu}{\hP^U_{\Delta t} \1}$ while, by Proposition \ref{prop: discrete stability}, the desired weighted average of $\Delta t \ \hlambda_{\Delta t}$ equals $\log \pair{\Phi^U_{\Delta t} \hmu^{\star}_{U, \Delta t}}{\hP^U_{\Delta t} \1}$.
This suggests that the first terms could lead to an error of $\mathcal{O}(1/T)$ in our approximation if that initial measure $\mu$ is $\mathcal{O}(1)$ away from $\hmu^{\star}_{U, \Delta t}$.\footnote{This is indeed the case in Examples \ref{example: convergence_test_single} and \ref{example: convergence_test_double} below.}
We now introduce two techniques to alleviate this issue.

The first technique is the so-called \emph{burn-in} procedure, in which we altogether drop from the sum the terms with $n\Delta t < t$, and reweigh the sum accordingly. In other words, we choose some $t>0$ (typically a function of $T$) and replace \eqref{eq: numerical approximation of principal eigenvalue} with
\begin{align*}
\hlambda^{\varepsilon, \alpha}_{\Delta t} = \frac{1}{T - t} \sum_{n = \left\lceil \tfrac t{\Delta t}\right\rceil}^{N - 1} \hlambda^{n}.
\end{align*}
This is equivalent to changing $T$ for $T-t$ and $\mu$ for the empirical measure of particles at $t$, which should be closer than $\mathcal{O}(1)$ away from $\hmu^{\star}_{U, \Delta t}$ if $t$ is chosen large enough in view of Proposition \ref{prop: discrete stability} and Section \ref{ssec:empirical-meas}.

The second technique applies when computing $\hlambda^{\varepsilon, \alpha}_{\Delta t}$ at $\varepsilon = \varepsilon_1$ and $\varepsilon = \varepsilon_2$ with $\varepsilon_1 > \varepsilon_2$. It consists in using the final distribution of the particles for the computation at $\varepsilon = \varepsilon_1$ as the initial distribution of the particles for the computation at $\varepsilon = \varepsilon_2$. Recall that, once appropriately rescaled in $\varepsilon$, the logarithm of the invariant density for the respective problems should be close to each other when both $\varepsilon_1$ and $\varepsilon_2$ are small. In particular, peaks in the density should be located at the same key points for both $\varepsilon_1$ and $\varepsilon_2$.
While this fact is symmetric, there is another consideration that does rely on the fact that $\varepsilon_1 > \varepsilon_2$: it is typically the dynamics with the smallest noise that takes the most resources to correct the effect of the poorly chosen initial condition and hence benefits the most from a choice of initial condition that is informed by a previous computation. In addition, we point out that the aforementioned two techniques can be applied simultaneously.

\begin{remark}
The effectiveness of the burn-in procedure relies on the convergence rate of \eqref{eq: stability of Euler scheme} as $k \rightarrow \infty$, and it is more effective for faster convergence. On the other hand, the effectiveness of the second technique relies on the convergence rate of $\varepsilon \log p_{U, \varepsilon}^{\star}$ as $\varepsilon \rightarrow 0^+$ \cite{fleming1997asymptotics}, where $p_{U, \varepsilon}^{\star}$ is the invariant density corresponding to $\varepsilon$ (more discussions on this convergence can be found in Section \ref{subsec: numerical example in vanishing-noise limit}), and it is more effective for faster convergence. These two convergence rates depend closely on $V$ and $b$, but the study of this dependence is beyond the scope of this paper.
\end{remark}

\section{Numerical examples}
\label{sec: numerical examples}

We first focus on exploring the vanishing-noise limit of the principal eigenvalue and the rate function. Then, we perform the convergence tests with respect to the final time $T$ and the time step size $\Delta t$ supporting the convergence of IPM.

\subsection{The principal eigenvalue and the rate function in the vanishing-noise limit}
\label{subsec: numerical example in vanishing-noise limit}

The following computations in this subsection are performed on a high-performance computing cluster with 2 Intel Xeon Gold 6226R (16 Core) CPUs and 96GB RAM. We consider the computation of $\lambda^{\varepsilon, \alpha}$ for certain values of $\varepsilon$ and $\alpha$. In particular, we choose $\varepsilon = 0.1, 0.01, 0.001$. For each fixed $\varepsilon$, we let $\alpha \in \left[-\frac{1}{10}, \frac{11}{10} \right]$ and compute $\hlambda^{\varepsilon, \alpha}_{\Delta t}$ for $\alpha = -\frac{1}{10} + \frac{j}{31} \frac{12}{10}$ with $j = 0, 1, \ldots, 31$. The computation of $\hlambda^{\varepsilon, \alpha}_{\Delta t}$ for each $\varepsilon$ with $32$ different values of $\alpha$ is performed at the same time in parallel on the 32 cores of the CPUs.
For the numerical discretization of our method, we choose $M = 500\,000$ and $\Delta t = 2^{-8}$ in Algorithm \ref{algo: IPM}. Also, unless specified, the initial measure of the particles is chosen to be the standard multivariate Gaussian distribution.

\begin{example}
\label{example: 2D_single_well}
Consider
\begin{gather*}
V^{\textnormal{E1}}(x_1, x_2) = \frac{x_1^2 + x_2^2}{2} + \frac{x_1^4 + x_2^4}{8}, \\
b^{\textnormal{E1}}(x_1, x_2) = \pi^{-1} (\cos(\pi x_1) \sin(\pi x_2), - \sin(\pi x_1) \cos(\pi x_2)).
\end{gather*}
Note that $V^{\textnormal{E1}}$ has a global minimum point at $(0, 0)$ and no other critical points. For $\alpha$ in an open interval containing $[0,1]$, it can be shown \cite{raquepas2020large} that $\lambda^{\varepsilon,\alpha}$ converges as $\varepsilon \to 0^+$ to
\begin{align}
\label{eq:lim-E1}
\lambda^{0, \alpha}
= 1 - \sqrt{1 + 4 \alpha (1 - \alpha)}.
\end{align}
\end{example}

We choose $T = 1024$. We show the numerical eigenvalue $\hlambda^{\varepsilon, \alpha}_{\Delta t}$ in Figure \ref{fig: eigenvalue_2D_single_well}. In addition, the numerical rate function $\widehat{I}^{\varepsilon}_{\Delta t}(s)$ obtained by the Legendre transform of $\hlambda^{\varepsilon, \alpha}_{\Delta t}$ is shown in Figure \ref{fig: rate_func_2D_single_well}. Moreover, the empirical density of particles at $T$ with $\alpha \approx 0.6742$ is shown in Figure \ref{fig: inv_meas_2D_single_well}. It can be seen from Figure \ref{fig: inv_meas_2D_single_well} that the particles get more localized around the global minimum point $(0, 0)$ of $V^{\textnormal{E1}}$ as $\varepsilon \rightarrow 0^+$.

\begin{figure}[ht]
\centering
\begin{subfigure}{0.4\textwidth}
\includegraphics[width=\columnwidth]{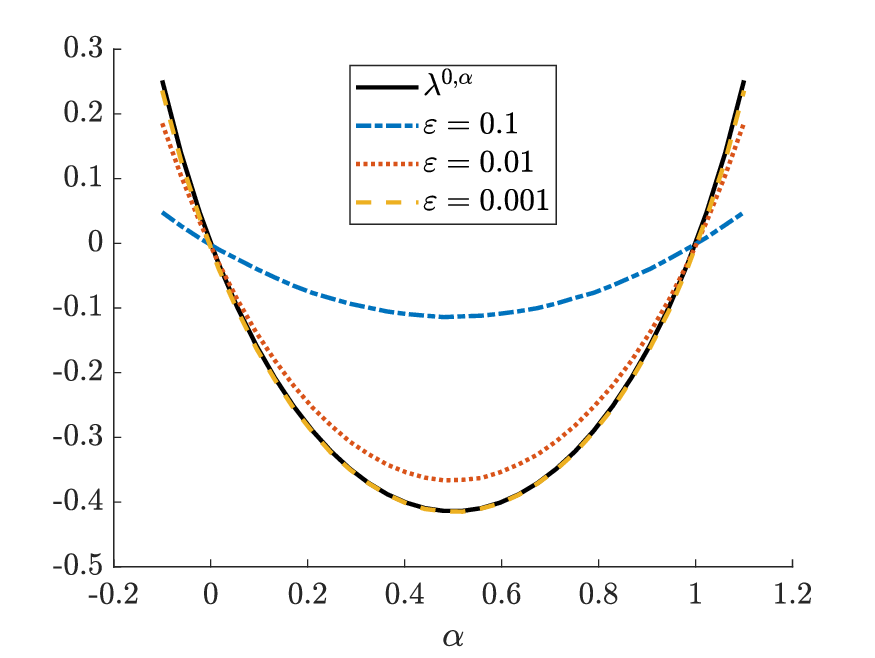}
\caption{$\hlambda^{\varepsilon, \alpha}_{\Delta t}$}
\label{fig: eigenvalue_2D_single_well}
\end{subfigure}
\begin{subfigure}{0.4\textwidth}
\includegraphics[width=\columnwidth]{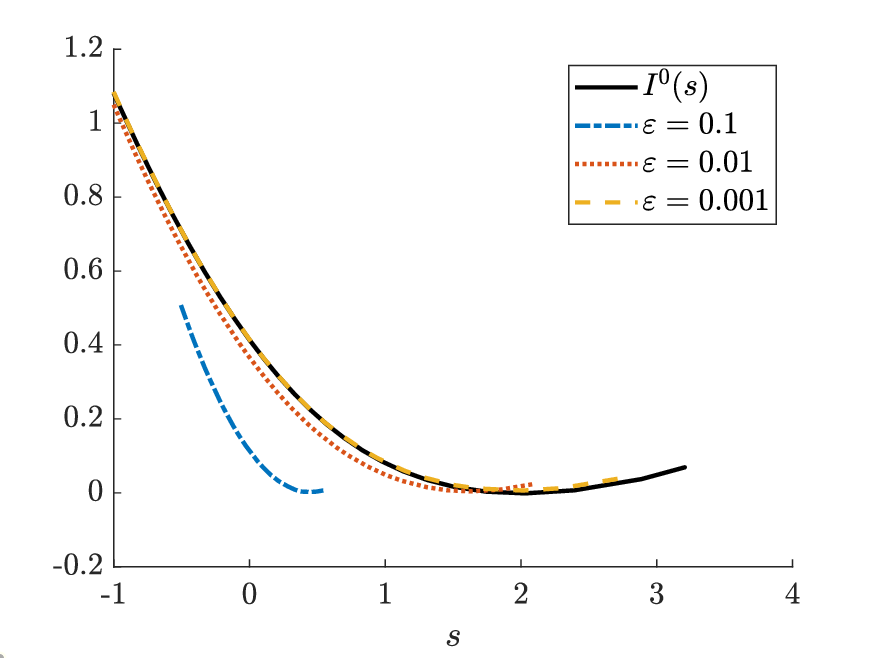}
\caption{$\widehat{I}^{\varepsilon}_{\Delta t}(s)$}
\label{fig: rate_func_2D_single_well}
\end{subfigure}
\caption{In Example \ref{example: 2D_single_well}, we plot our numerical approximation $\hlambda^{\varepsilon,\alpha}_{\Delta t}$ of the principal eigenvalue $\lambda^{\varepsilon,\alpha}$ and the resulting approximation $\widehat{I}^{\varepsilon}_{\Delta t}(s)$ of the rate function $I^\varepsilon(s)$, compared respectively to the limit $\lambda^{0,\alpha}$ in \eqref{eq:lim-E1} and its Legendre transform $I^0(s)$. Note the consistency of the symmetries mentioned in Section \ref{sec: intro}.  Also note that the restriction of $\widehat{I}^{\varepsilon}_{\Delta t}(s)$ to certain values of $s$ is due to our restriction of $\hlambda^{\varepsilon,\alpha}_{\Delta t}$ and how it interacts with the derivatives.}
\end{figure}
\begin{figure}[ht]
\centering
\begin{subfigure}{.33\textwidth}
\includegraphics[width=\columnwidth]{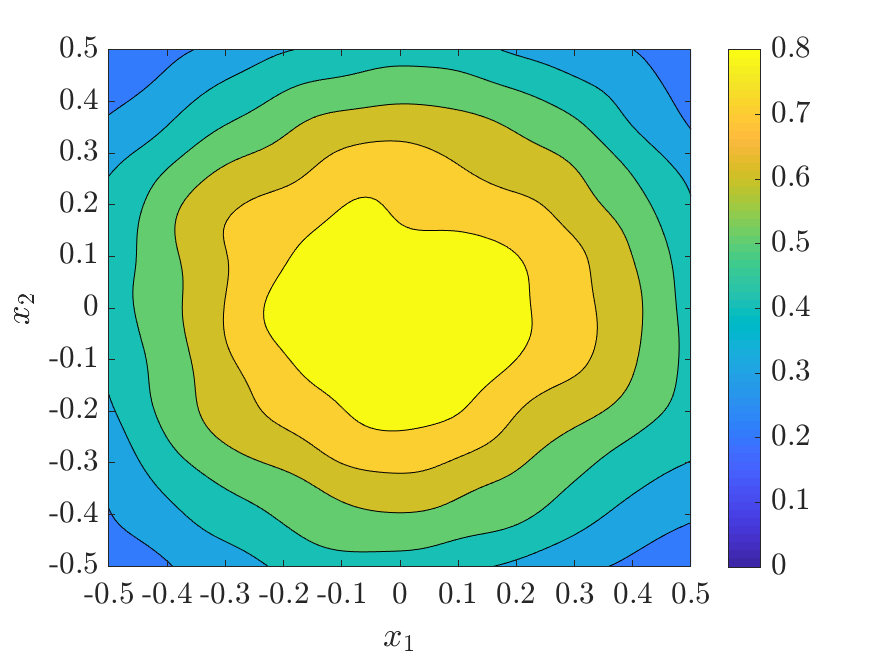}
\caption{$\varepsilon = 0.1$}
\end{subfigure}
\begin{subfigure}{.33\textwidth}
\includegraphics[width=\columnwidth]{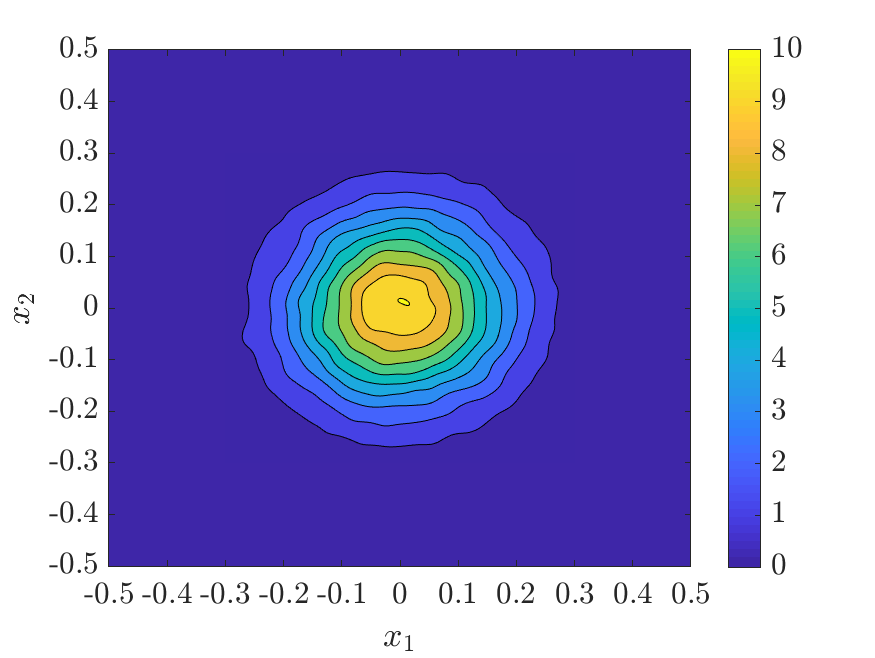}
\caption{$\varepsilon = 0.01$}
\end{subfigure}
\begin{subfigure}{.33\textwidth}
\includegraphics[width=\columnwidth]{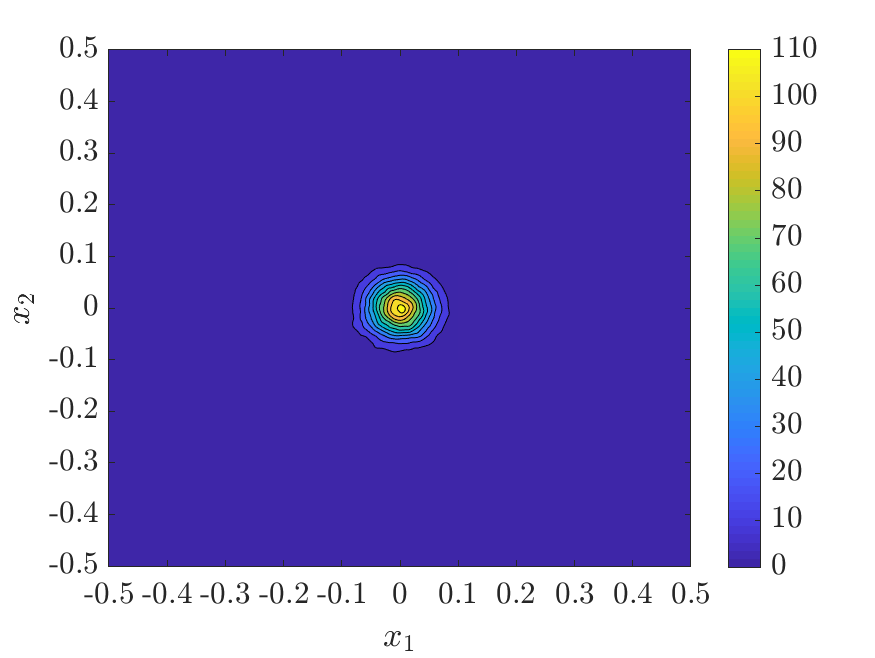}
\caption{$\varepsilon = 0.001$}
\end{subfigure}
\caption{In Example \ref{example: 2D_single_well}, we plot the empirical density of particles at $T$ with $\alpha \approx 0.6742$. {Note the concentration of the mass of the measure around $(0, 0)$ as $\varepsilon$ decreases.}}
\label{fig: inv_meas_2D_single_well}
\end{figure}

\begin{example}
\label{example: 2D_double_well}
    Consider
    \begin{align*}
    V^{\textnormal{E2}}(x_1, x_2; a) = x_1^4 - 2 x_1^2 + (1 + a (x_1 - 1)^2) x_2^2 + x_2^4, \quad b^{\textnormal{E2}}(x_1, x_2) = b^{\textnormal{E1}}(x_1,x_2),
    \end{align*}
    with $a = 0.4$. Note that $V^{\textnormal{E2}}$ has two local minima at $(-1, 0)$ and $(1, 0)$, as well as a saddle point at $(0,0)$.
    For $\alpha$ in an open interval containing $[0,1]$, it can be shown \cite{raquepas2020large} that $\lambda^{\varepsilon,\alpha}$ converges as $\varepsilon \to 0^+$ to
    \begin{align} \label{eq: eigenvalue_2D_double_well}
    {\lambda^{0, \alpha}}
    \widetilde{\lambda}^{\alpha}(a)
    = \max(\lambda^{\alpha}_{+}, \lambda^{\alpha}_{-}),
    \end{align}
    where
    $
    \lambda^{\alpha}_{\pm} = - \Tr X_{\pm}(\alpha) + \frac{1}{2} \Tr D^2 V^{\textnormal{E2}}|_{(\pm 1, 0)},
    $
    with $X_{\pm}(\alpha)$ satisfying the algebraic Riccati equation
    \begin{align*}
    0 &= X_{\pm}(\alpha)^2 - \frac{1 - 2\alpha}{2} (\nabla b|_{(\pm1, 0)}^{\top} X_{\pm}(\alpha) + X_{\pm}(\alpha)^{\top} \nabla b|_{(\pm1, 0)}) \\
    &\qquad\qquad {}- \frac{1}{4} D^2 V|_{(\pm1, 0)} D^2 V|_{(\pm1, 0)} - \alpha(1 - \alpha) \nabla b|_{(\pm1, 0)}^{\top} \nabla b|_{(\pm1, 0)} \\
    &\qquad\qquad {}+ \frac{1}{4}(\nabla b|_{(\pm1, 0)}^{\top} D^2 V|_{(\pm 1, 0)} + D^2 V|_{(\pm 1, 0)} \nabla b|_{(\pm1, 0)})
    \end{align*}
    We have omitted the superscripts ``${\textnormal{E2}}$'' and the parameter ``$a$'' in this last equation to avoid cluttering the notation. This equation involving 2-by-2 matrices is easily solved numerically.
\end{example}

We choose $T = 2048$. We use the burn-in procedure, in which we start computing the eigenvalue from $t = 1024$. We show $\hlambda^{\varepsilon, \alpha}_{\Delta t}$ in Figure \ref{fig: eigenvalue_2D_double_well} and $\widehat{I}^{\varepsilon}_{\Delta t}(s)$ in Figure \ref{fig: rate_func_2D_double_well}. The empirical density of particles at $T$ with $\alpha \approx 0.5968$ is shown in Figure \ref{fig: inv_meas_2D_double_well_alpha_18} and that with $\alpha \approx 1.0613$ is shown in Figure \ref{fig: inv_meas_2D_double_well_alpha_30}. We can see from Figures \ref{fig: inv_meas_2D_double_well_alpha_18} and \ref{fig: inv_meas_2D_double_well_alpha_30} that the particles are localized around different local minimum points of $V^{\textnormal{E2}}$ for different values of $\alpha$.

\begin{figure}[ht]
\centering
\begin{subfigure}{0.4\textwidth}
\includegraphics[width=\columnwidth]{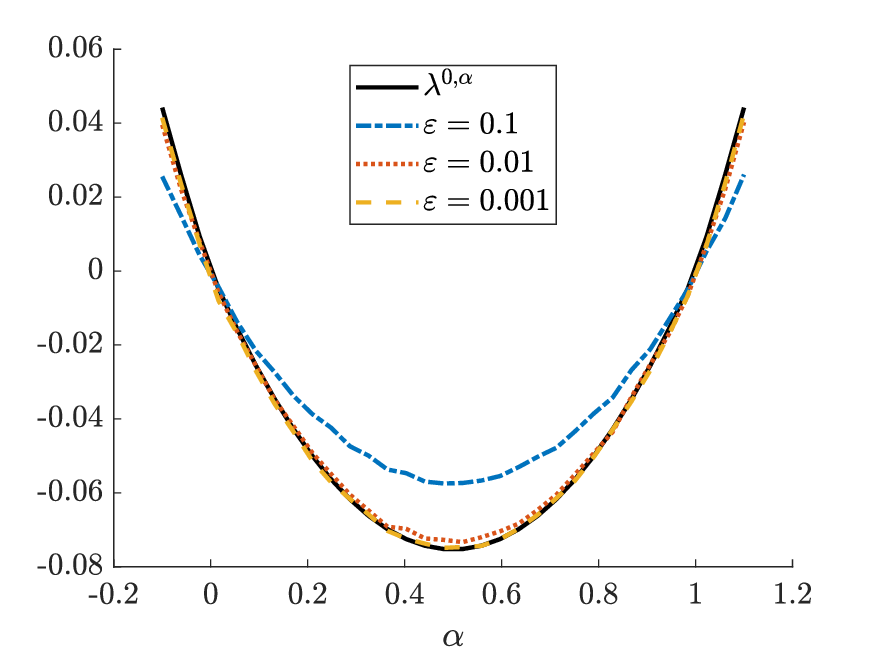}
\caption{$\hlambda^{\varepsilon, \alpha}_{\Delta t}$}
\label{fig: eigenvalue_2D_double_well}
\end{subfigure}
\begin{subfigure}{0.4\textwidth}
\includegraphics[width=\columnwidth]{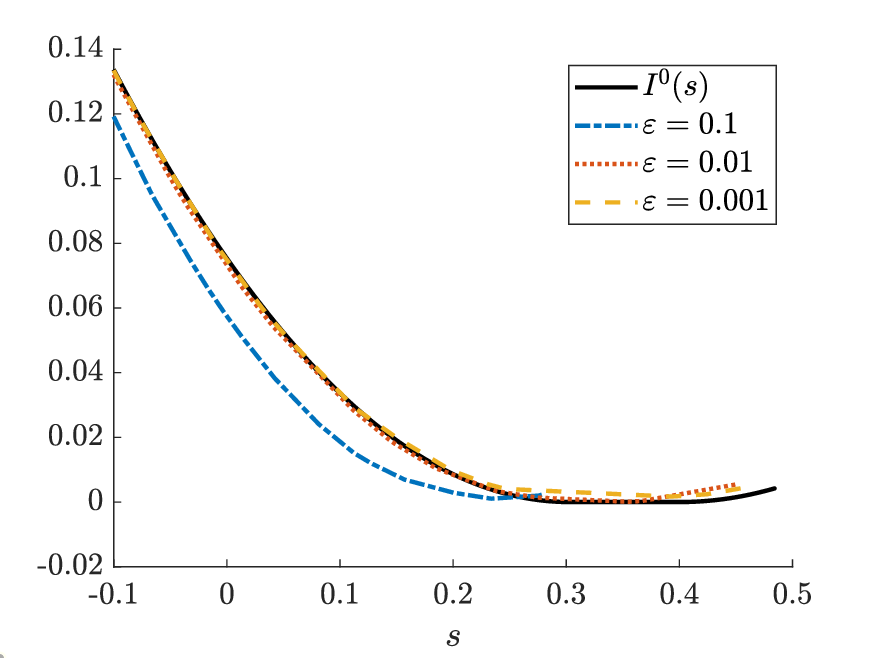}
\caption{$\widehat{I}^{\varepsilon}_{\Delta t}(s)$}
\label{fig: rate_func_2D_double_well}
\end{subfigure}
\caption{In Example \ref{example: 2D_double_well}, we plot our numerical approximation $\hlambda^{\varepsilon,\alpha}_{\Delta t}$ of the principal eigenvalue $\lambda^{\varepsilon,\alpha}$ and the resulting approximation $\widehat{I}^{\varepsilon}_{\Delta t}(s)$ of the rate function $I^\varepsilon(s)$, compared respectively to the limit $\lambda^{0,\alpha}$ in \eqref{eq: eigenvalue_2D_double_well} and its Legendre transform $I^0(s)$. Note that the maximum in \eqref{eq: eigenvalue_2D_double_well} causes a discontinuity of the derivative of the limit of the eigenvalue in $\alpha = 0$ and $\alpha = 1$, in turn causing flat regions in the limit of the rate function.}
\end{figure}
\begin{figure}[ht]
\centering
\begin{subfigure}{.33\textwidth}
\includegraphics[width=\columnwidth]{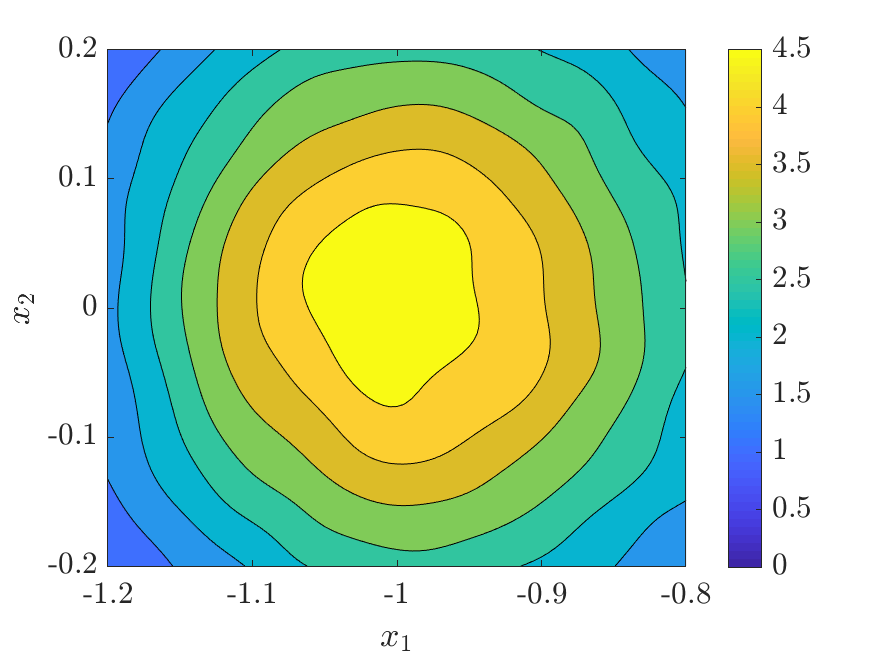}
\caption{$\varepsilon = 0.1$}
\end{subfigure}
\begin{subfigure}{.33\textwidth}
\includegraphics[width=\columnwidth]{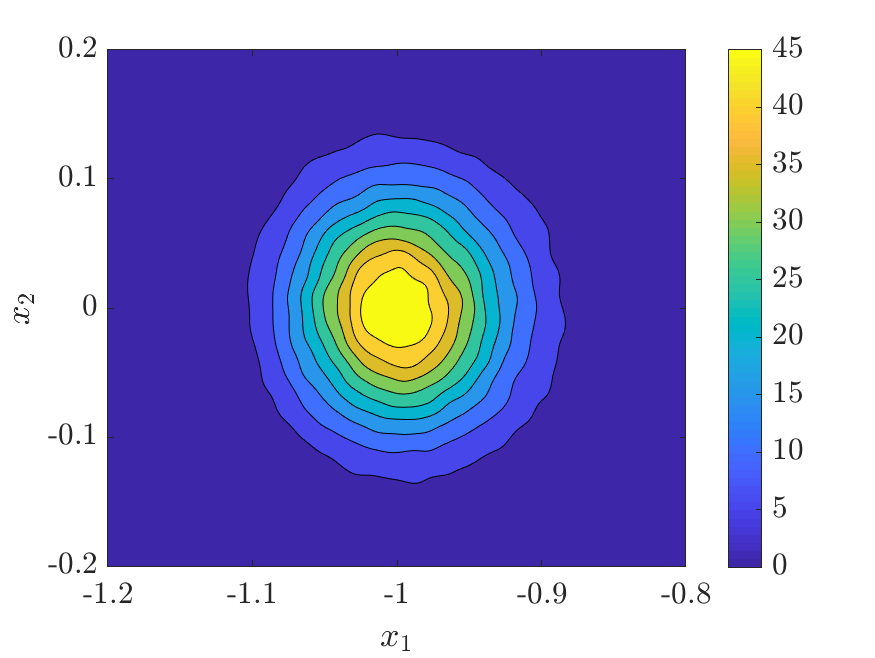}
\caption{$\varepsilon = 0.01$}
\end{subfigure}
\begin{subfigure}{.33\textwidth}
\includegraphics[width=\columnwidth]{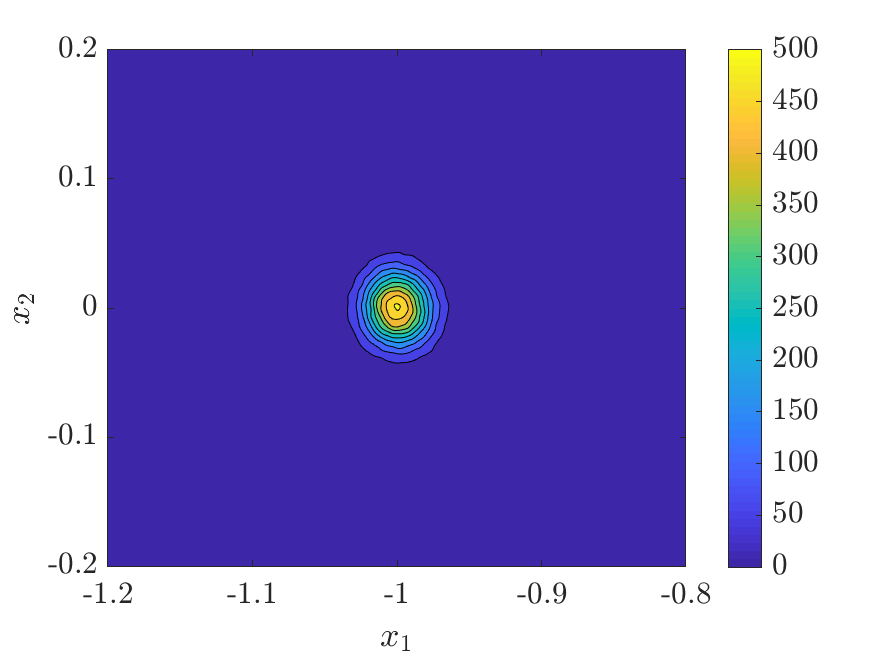}
\caption{$\varepsilon = 0.001$}
\end{subfigure}
\caption{In Example \ref{example: 2D_double_well}, we plot the empirical density of particles at $T$ with $\alpha \approx 0.5968$. Note the concentration of the mass of the measure around $(-1, 0)$ as $\varepsilon$ decreases; no mass could be observed near $(1,0)$ at $\varepsilon = 0.001$.}
\label{fig: inv_meas_2D_double_well_alpha_18}
\end{figure}
\begin{figure}[ht]
\centering
\begin{subfigure}{.33\textwidth}
\includegraphics[width=\columnwidth]{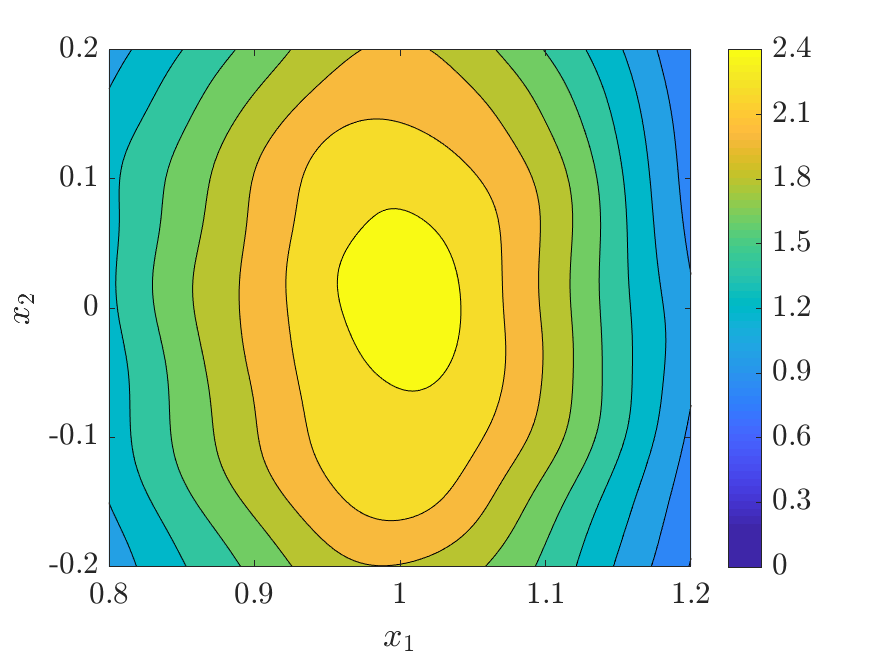}
\caption{$\varepsilon = 0.1$}
\end{subfigure}
\begin{subfigure}{.33\textwidth}
\includegraphics[width=\columnwidth]{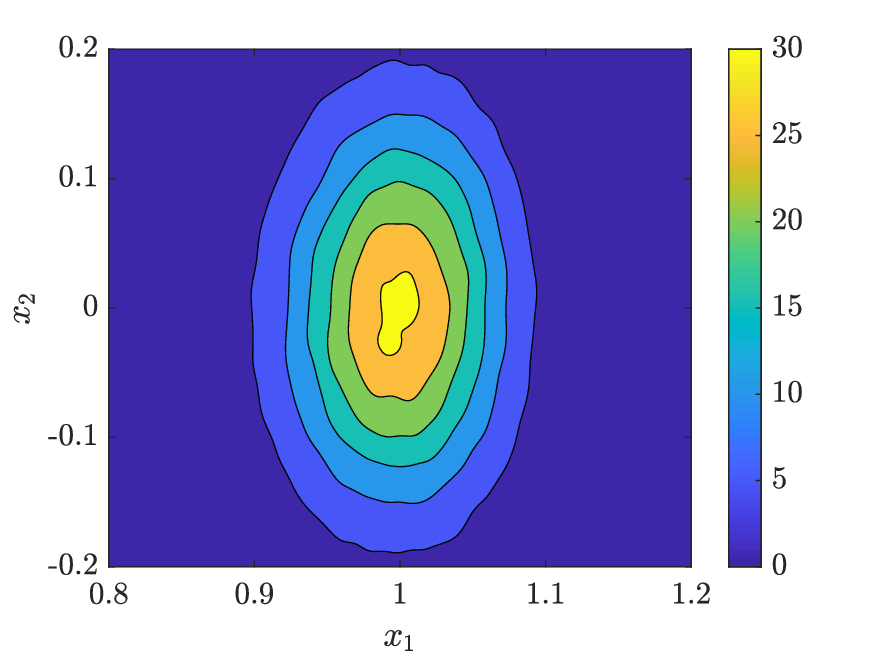}
\caption{$\varepsilon = 0.01$}
\end{subfigure}
\begin{subfigure}{.33\textwidth}
\includegraphics[width=\columnwidth]{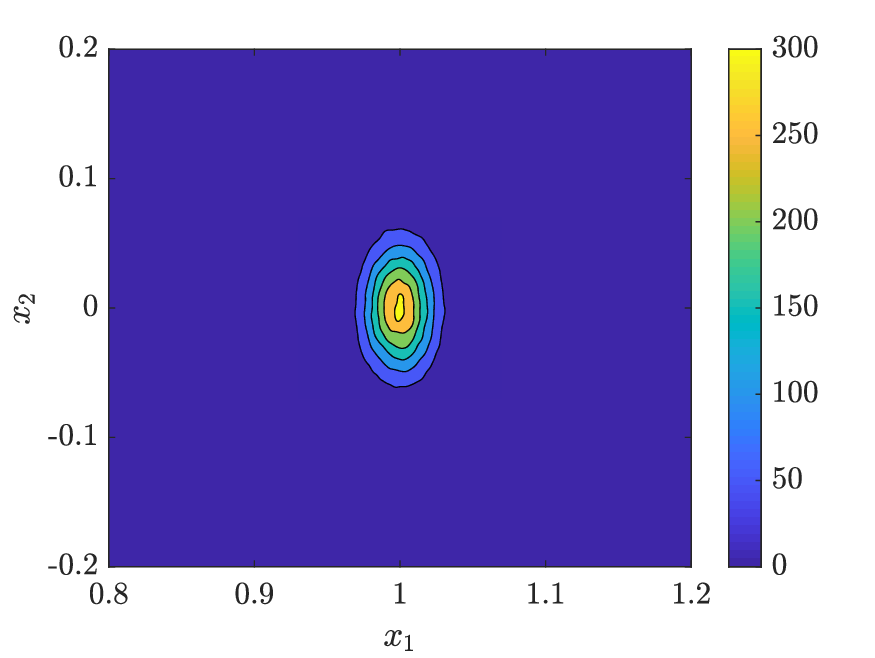}
\caption{$\varepsilon = 0.001$}
\end{subfigure}
\caption{In Example \ref{example: 2D_double_well}, we plot the empirical density of particles at $T$ with $\alpha \approx 1.0613$. Note the concentration of the mass of the measure around $(1, 0)$ as $\varepsilon$ decreases; no mass could be observed near $(-1,0)$ at $\varepsilon = 0.001$.}
\label{fig: inv_meas_2D_double_well_alpha_30}
\end{figure}

\begin{remark}
The potential $V^{\textnormal{E2}}$ has two local minima, so the invariant density might be bimodal when $\varepsilon$ is moderately large, with the mass of the measure being concentrated around the two different local minimum points of $V^{\textnormal{E2}}$. The empirical density obtained by our method captures this feature; see Figure \ref{fig: inv_meas_2D_double_well_alpha_30_bimodal}. This shows that our IPM can accurately capture the shape of multimodal invariant measures, which is known to be difficult for some sampling methods, e.g., MCMC methods.
\end{remark}

\begin{figure}[ht]
  \centering
  \includegraphics[width=0.4\textwidth]{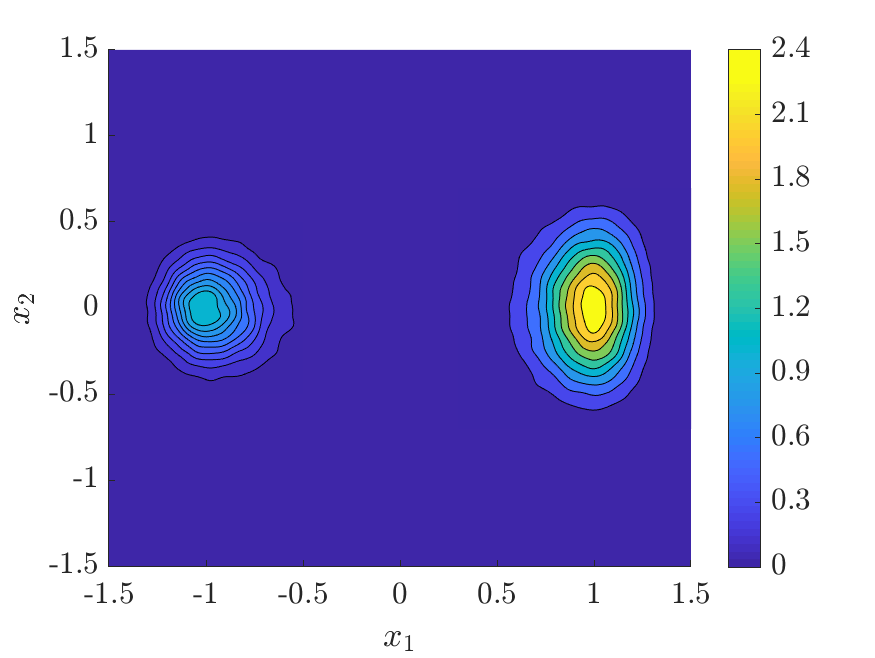}
  \caption{In Example \ref{example: 2D_double_well}, we plot the empirical density of particles at $T$ for $\varepsilon = 0.1$ with $\alpha \approx 1.0613$. Note that the density is bimodal, with the mass of measure being concentrated around $(-1, 0)$ and $(1, 0)$, the two local minimum points of $V^\textnormal{E2}$.}
  \label{fig: inv_meas_2D_double_well_alpha_30_bimodal}
\end{figure}

\begin{example}
\label{example: 16D_coupled}
Consider
\begin{align*}
V^\textnormal{E3}(x_1, \ldots, x_{16}) = \frac{1}{2} x^{\top} M x + 4 |x|^4, \quad b^\textnormal{E3}(x_1, \ldots, x_{16}) = \eta(|x|) B x,
\end{align*}
where $x = (x_1, \ldots, x_{16})^{\top} \in \mathbb{R}^{16}$, $\eta$ is a smooth cut-off function with $\eta(r) = 1$ for $r \leq 1$ and $\eta(r) = 0$ for $r \geq 2$,
\begin{align*}
M = 
\begin{pmatrix}
 5 & 0 & \dotsb & 0\\
 0 & 6 & \dotsb & 0 \\
 \vdots & \vdots & \ddots & \vdots \\
 0 & 0 & \dotsb & 20
\end{pmatrix}
 \in \mathbb{R}^{16 \times 16}, \quad
B = Q^{\top} 
\begin{pmatrix}
 B_1 & 0 & \dotsb & 0 \\
 0 & B_2 & \dotsb & 0 \\
 \vdots &\vdots & \ddots &\vdots \\
 0 & 0 & \dotsb & B_8
\end{pmatrix}
 Q \in \mathbb{R}^{16 \times 16}, 
\end{align*}
with
\begin{align*}
B_k = 
\begin{pmatrix}
 0 & 1 \\
 -1 & 0 \\
\end{pmatrix}
, \quad k = 1, \ldots, 8,
\end{align*}
and $Q \in \mathbb{R}^{16 \times 16}$ a random orthogonal matrix sampled using the method in \cite{mezzadri2006generate}. With this choice, $B$ does not have any particular block structure, there is no obvious commutation relation between $M$ and $B$, and all 16 coordinates, $x_1, \dotsc, x_{16},$ are genuinely coupled. Note that $V^\textnormal{E3}$ has a global minimum point at the origin and no other critical points. For $\alpha$ in an open interval containing $[0,1]$, it can be shown \cite{raquepas2020large} that $\lambda^{\varepsilon,\alpha}$ converges as $\varepsilon \to 0^+$ to
\begin{align} \label{eq: eigenvalue_16D_coupled}
\lambda^{0, \alpha} = -\Tr X(\alpha) + \frac{1}{2} \Tr M - \alpha \Tr B,
\end{align}
where $X(\alpha)$ is the maximal solution to
\begin{align*}
X(\alpha)^\top X(\alpha) - \frac{1}{2} (1-2\alpha) (B^{\top} X(\alpha) + X(\alpha)^{\top} B) - \frac{1}{4} M^\top M + \frac{1}{4} (B^\top M + M^\top B) - \alpha(1-\alpha) B^\top B = 0.
\end{align*}
The above equation can be easily solved numerically.
\end{example}

We choose $T = 2048$. We use the burn-in procedure for $\varepsilon = 0.1, 0.01$, in which we start computing the eigenvalue from $t = 1024$. For $\varepsilon = 0.001$, we use the empirical measure of particles at $T$ obtained at $\varepsilon = 0.01$ as the initial measure. We show $\hlambda^{\varepsilon, \alpha}_{\Delta t}$ in Figure \ref{fig: eigenvalue_16D_coupled} and $\widehat{I}^{\varepsilon}_{\Delta t}(s)$ in Figure \ref{fig: rate_func_16D_coupled}. The 2-dimensional marginal empirical density of $(x_{8}, x_{9})$ of particles at $T$ with $\alpha \approx 0.3258$ are shown in Figure \ref{fig: inv_meas_16D_coupled}. It can be seen from Figure \ref{fig: inv_meas_16D_coupled} that the particles get more localized around the origin as $\varepsilon \rightarrow 0^+$.

\begin{figure}[ht]
\centering
\begin{subfigure}{0.4\textwidth}
\includegraphics[width=\columnwidth]{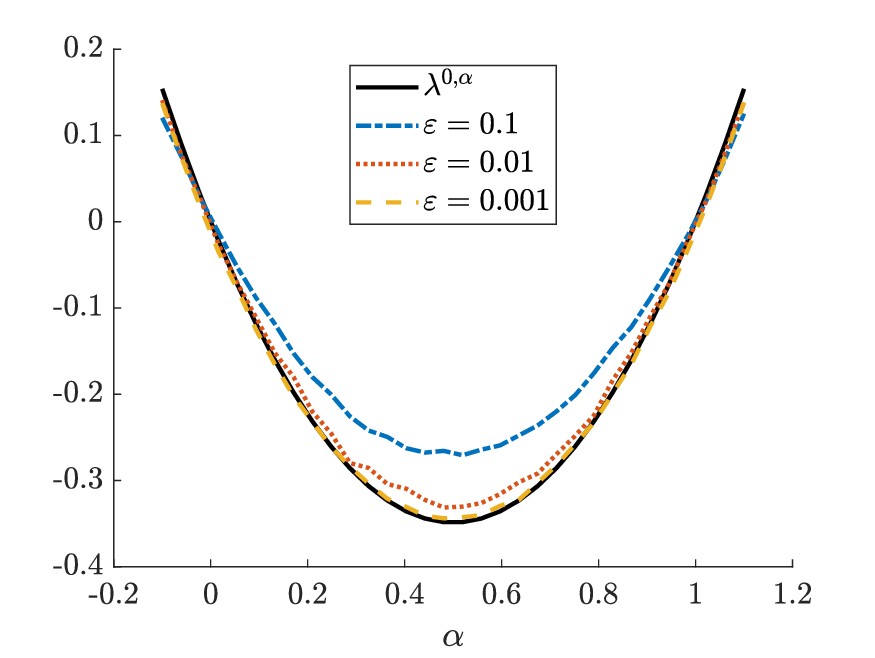}
\caption{$\hlambda^{\varepsilon, \alpha}_{\Delta t}$}
\label{fig: eigenvalue_16D_coupled}
\end{subfigure}
\begin{subfigure}{0.4\textwidth}
\includegraphics[width=\columnwidth]{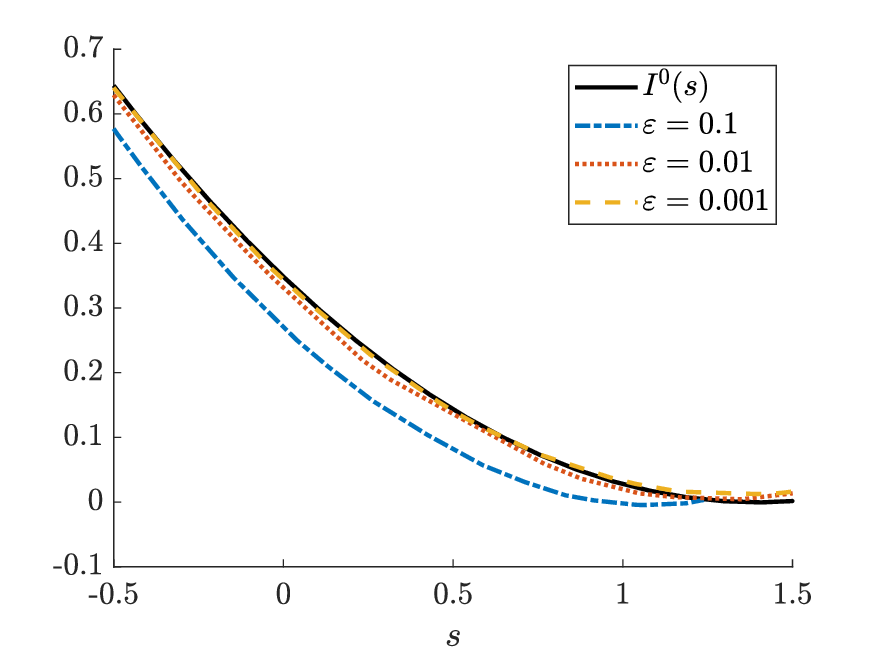}
\caption{$\widehat{I}^{\varepsilon}_{\Delta t}(s)$}
\label{fig: rate_func_16D_coupled}
\end{subfigure}
\caption{In Example \ref{example: 16D_coupled}, we plot our numerical approximation $\hlambda^{\varepsilon,\alpha}_{\Delta t}$ of the principal eigenvalue $\lambda^{\varepsilon,\alpha}$ and the resulting approximation $\widehat{I}^{\varepsilon}_{\Delta t}(s)$ of the rate function $I^\varepsilon(s)$, compared respectively to the limit $\lambda^{0,\alpha}$ in \eqref{eq: eigenvalue_16D_coupled} and its Legendre transform $I^0(s)$.}
\end{figure}

\begin{figure}[ht]
\centering
\begin{subfigure}{.33\textwidth}
\includegraphics[width=\columnwidth]{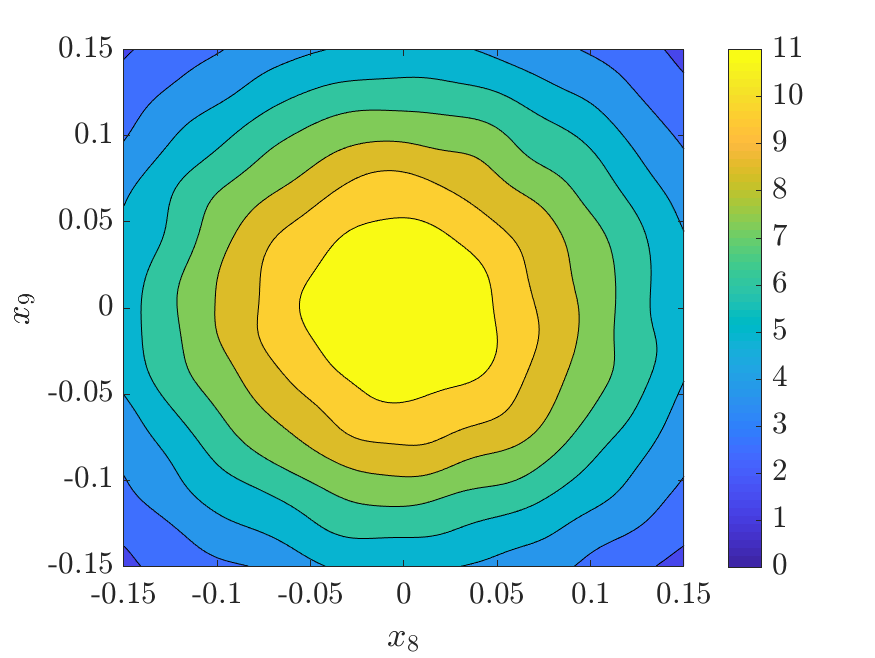}
\caption{$\varepsilon = 0.1$}
\end{subfigure}
\begin{subfigure}{.33\textwidth}
\includegraphics[width=\columnwidth]{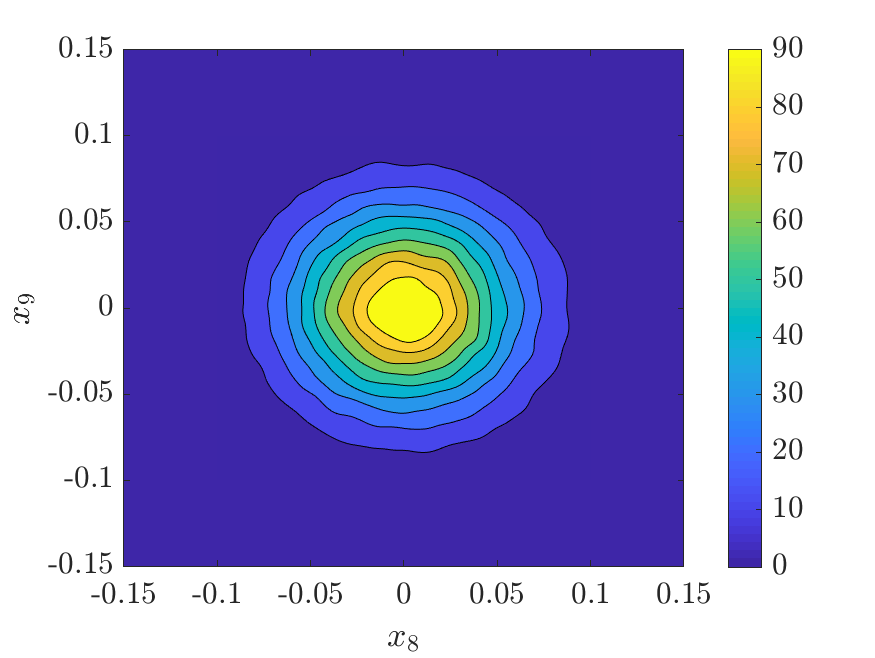}
\caption{$\varepsilon = 0.01$}
\end{subfigure}
\begin{subfigure}{.33\textwidth}
\includegraphics[width=\columnwidth]{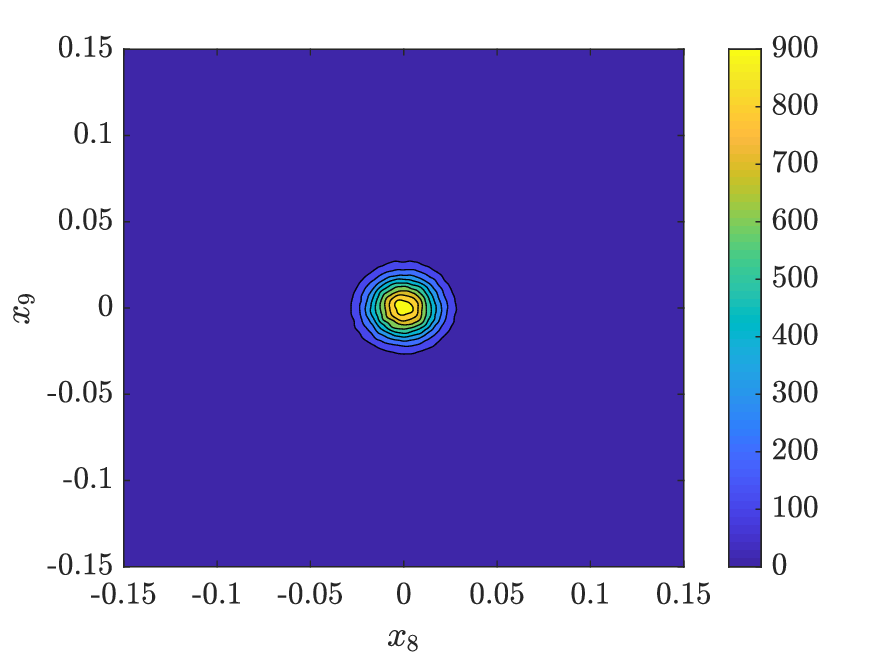}
\caption{$\varepsilon = 0.001$}
\end{subfigure}
\caption{In Example \ref{example: 16D_coupled}, we plot the 2-dimensional marginal empirical density of $(x_{8}, x_{9})$ of particles at $T$ with $\alpha \approx 0.3258$. Note the concentration of the mass of the measure around $(0, 0)$ as $\varepsilon$ decreases.}
\label{fig: inv_meas_16D_coupled}
\end{figure}

From the above examples of both small and large $d$, we can observe within visual tolerance the convergence of both the numerical principal eigenvalue $\hlambda^{\varepsilon, \alpha}_{\Delta t}$ and the numerical rate function $\widehat{I}^{\varepsilon}_{\Delta t}(s)$ to their respective analytical vanishing-noise limits $\lambda^{0,\alpha}$ and $I^0(s)$, with a fixed number of particles and a fixed time step size. Furthermore, the maximum of the 2-dimensional (marginal) empirical density of particles at $T$ is approximately proportional to $\varepsilon^{-1}$ for small $\varepsilon$.
We know that the invariant density $p^{\star}_{U, \varepsilon}$ of $\mu^{\star}_{U,\varepsilon}$ is the (suitably normalized) principal eigenfunction $\psi^{\varepsilon, \alpha}$ for the dual of the operator $\cL^{\varepsilon, \alpha} + U^{\varepsilon, \alpha}$. Under certain additional conditions, the study in \cite{fleming1997asymptotics} shows that $\varepsilon \log \psi^{\varepsilon, \alpha}$ has a non-trivial limit as $\varepsilon \rightarrow 0^+$: the density is asymptotically proportional to $\exp(-\varepsilon^{-1} \Phi)$ for some function $\Phi$, with a normalizing constant that is asymptotically $\mathcal{O}(\varepsilon^{- d / 2})$. Hence, the observed asymptotic behavior of the empirical density of particles at $T$ as $\varepsilon \rightarrow 0^+$ is consistent with the theory in \cite{fleming1997asymptotics}.

We end this subsection with a discussion on an example where the assumptions \eqref{eq:ND-on-V}--\eqref{eq:ND-on-b-vs-V} fail, preventing us from appealing to the proof of \cite{raquepas2020large} for convergence in the limit $\varepsilon \to 0^+$. In such situations, it is possible for $\lambda^{\varepsilon, \alpha}$ to diverge as $\varepsilon \rightarrow 0^+$.

\begin{example} \label{example: 2D_circle}
Consider
\begin{gather*}
V^{\textnormal{E4}}(x_1, x_2) = - \frac{x_1^2 + x_2^2}{4} + \frac{x_1^4 + 2 x_1^2 x_2^2 + x_2^4}{8}, \\
b^{\textnormal{E4}}(x_1, x_2) = (\cos(x_1) \sin(x_2), - \sin(x_1) \cos(x_2)).
\end{gather*}
Note that the $\nabla V^{\textnormal{E4}}(x) = 0$ for all $x$ on the circle $\{(x_1, x_2): x_1^2 + x_2^2=1\}$, whereas $b^{\textnormal{E4}}$ acts non-trivially along that circle. In Figures \ref{fig: eigenvalue_zoom_in_2D_circle} and \ref{fig: eigenvalue_rescaled_2D_circle}, we see that the eigenvalue is of different orders in $\varepsilon$ depending on whether $\alpha\in[0,1]$ or $\alpha\notin[0,1]$.
\end{example}

\begin{figure}[ht]
\centering
\begin{subfigure}{0.4\textwidth}
\includegraphics[width=\columnwidth]{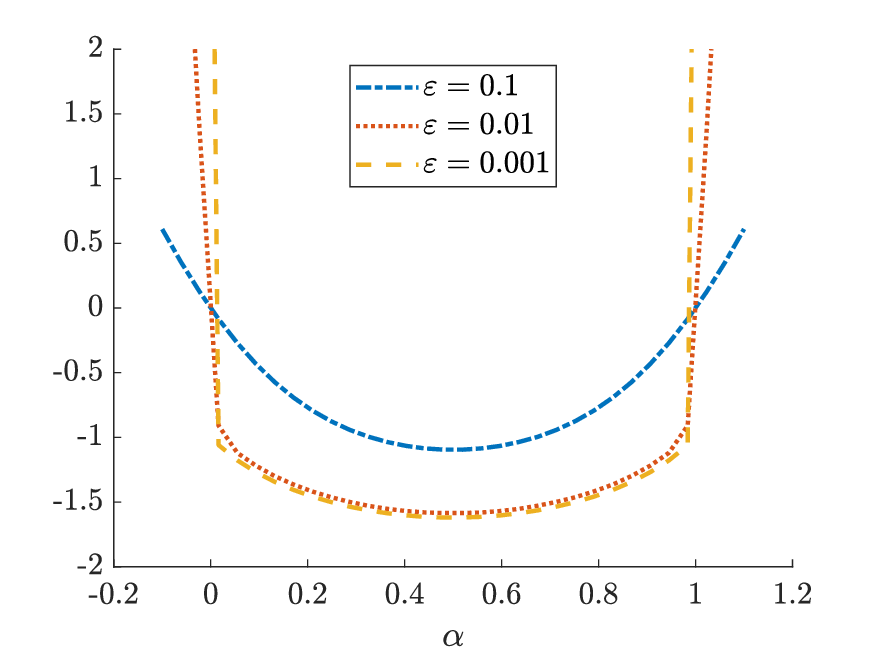}
\caption{$\hlambda^{\varepsilon, \alpha}_{\Delta t}$}
\label{fig: eigenvalue_zoom_in_2D_circle}
\end{subfigure}
\begin{subfigure}{0.4\textwidth}
\includegraphics[width=\columnwidth]{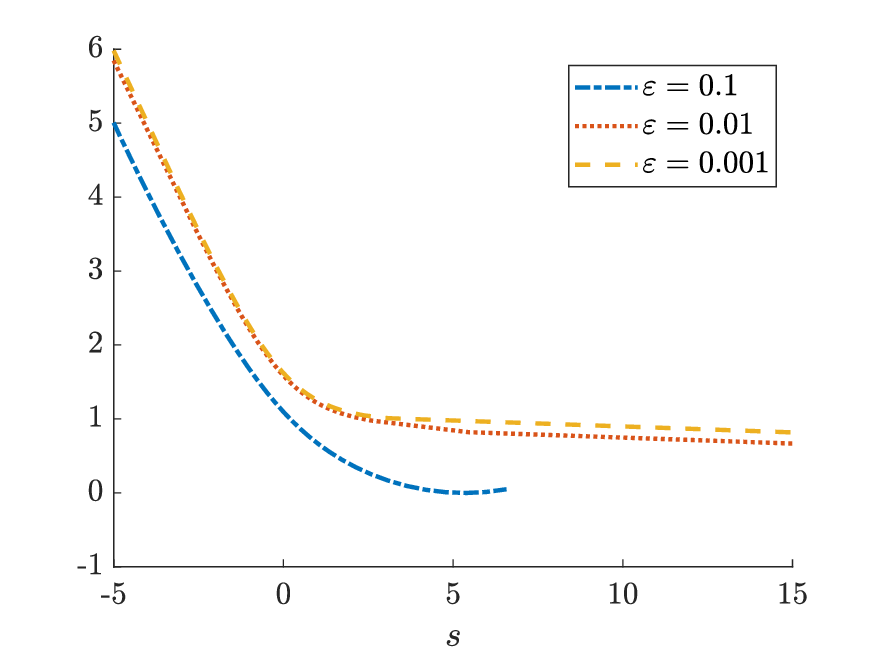}
\caption{$\widehat{I}^{\varepsilon}_{\Delta t}(s)$}
\label{fig: rate_func_zoom_in_2D_circle}
\end{subfigure}
\caption{In Example \ref{example: 2D_circle}, we plot our numerical approximation $\hlambda^{\varepsilon,\alpha}_{\Delta t}$ of the principal eigenvalue $\lambda^{\varepsilon,\alpha}$ and the resulting approximation $\widehat{I}^{\varepsilon}_{\Delta t}(s)$ of the rate function $I^\varepsilon(s)$.}
\end{figure}
\begin{figure}[ht]
\centering
\begin{subfigure}{0.4\textwidth}
\includegraphics[width=\columnwidth]{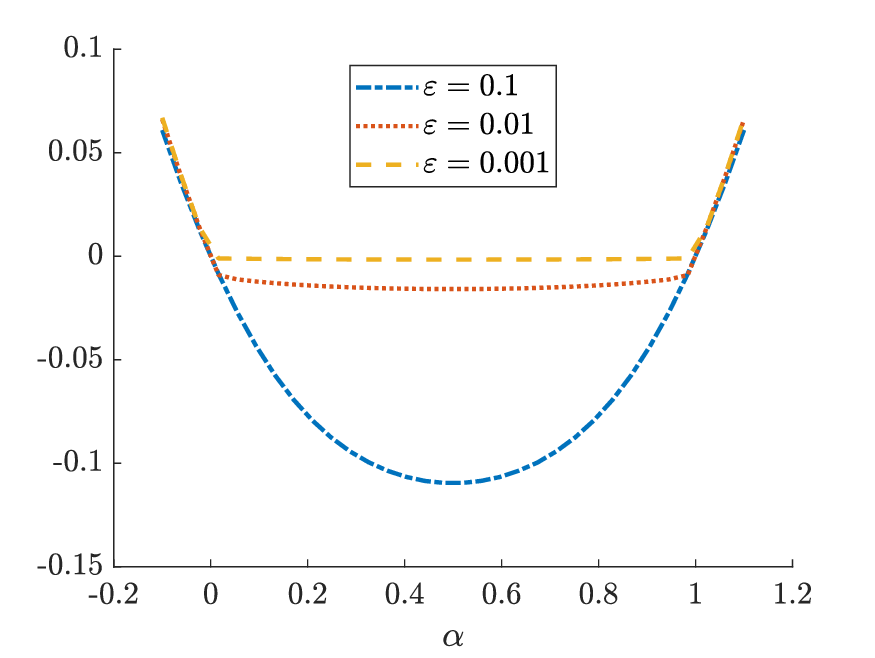}
\caption{$\varepsilon \hlambda^{\varepsilon, \alpha}_{\Delta t}$}
\label{fig: eigenvalue_rescaled_2D_circle}
\end{subfigure}
\begin{subfigure}{0.4\textwidth}
\includegraphics[width=\columnwidth]{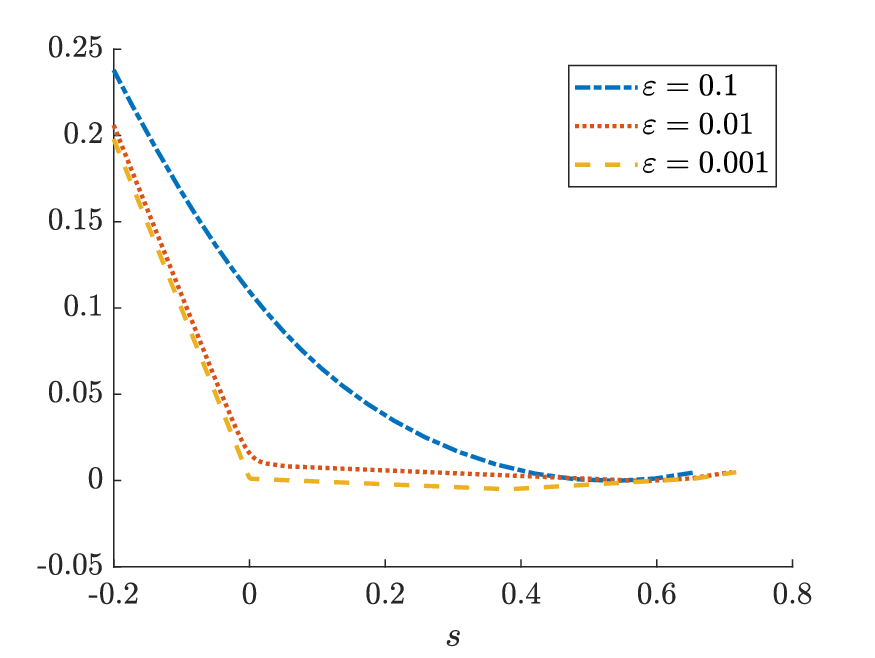}
\caption{$\varepsilon\widehat{I}^{\varepsilon}_{\Delta t}(\varepsilon^{-1}s)$}
\label{fig: rate_func_rescaled_2D_circle}
\end{subfigure}
\caption{In Example \ref{example: 2D_circle}, we plot our numerical approximation $\varepsilon \hlambda^{\varepsilon,\alpha}_{\Delta t}$ of the rescaled principal eigenvalue $\varepsilon \lambda^{\varepsilon,\alpha}$ and the resulting approximation $\varepsilon\widehat{I}^{\varepsilon}_{\Delta t}(\varepsilon^{-1}s)$ of the rescaled rate function $\varepsilon I^{\varepsilon}(\varepsilon^{-1}s)$.}
\end{figure}

For this example, we use $T = 1024$. We show $\hlambda^{\varepsilon, \alpha}_{\Delta t}$ in Figure \ref{fig: eigenvalue_zoom_in_2D_circle} and its Legendre transform $\widehat{I}^{\varepsilon}_{\Delta t}(s)$ in Figure \ref{fig: rate_func_zoom_in_2D_circle}. In particular, the zero of $\widehat{I}^{\varepsilon}_{\Delta t}$\,---\, which is the mean entropy production rate for that value of $\varepsilon$\,---\, seems to diverge as $\varepsilon \to 0^+$, as expected due to the inverse power of $\varepsilon$ in the definition \eqref{eq:def-Seps} of the entropy production and the periodic orbit of the deterministic dynamics along which the work done by $b$ per unit time is non-zero.
We also show $\varepsilon\hlambda^{\varepsilon, \alpha}_{\Delta t}$ in Figure \ref{fig: eigenvalue_rescaled_2D_circle} and its Legendre transform $\varepsilon\widehat{I}^{\varepsilon}_{\Delta t}(\varepsilon^{-1}s)$ in Figure \ref{fig: rate_func_rescaled_2D_circle}, as studied in \cite{BDG15,BGL22}. In particular, a key feature discussed in \cite[Section 5]{BDG15} is emerging as $\varepsilon \to 0^+$: a kink in $\varepsilon\widehat{I}^{\varepsilon}_{\Delta t}(\varepsilon^{-1}s)$ at $s=0$, where two flat regions meet at an angle compatible with the Gallavotti--Cohen symmetry. {The example also confirms that in some (but not all) scenarios, the limits of $I^\varepsilon$ and $\varepsilon I^{\varepsilon}(\varepsilon^{-1}\,\cdot\,)$ provide complementary, non-trivial information on the fluctuations of $S^\varepsilon_t$.}

\subsection{Convergence tests}

Recall that our numerical discretization using the operator splitting scheme and the Euler--Maruyama scheme converges with respect to the final time $T$ as shown in Proposition \ref{prop: discrete stability}, and that it also converges with respect to the time step size $\Delta t$ as shown in Theorem \ref{thm: convergence of IPM wrt time step size}. In this subsection, we consider the ensuing IPM on two examples with a quadratic potential and a linear drift, both of which admit explicit theoretical expressions for the principal eigenvalue $\lambda^{\varepsilon, \alpha}$ that are independent of $\varepsilon$. We perform convergence tests with respect to $T$ and $\Delta t$, respectively, by comparing $| \hlambda^{\varepsilon,\alpha}_{\Delta t} - \lambda^{\varepsilon,\alpha}|$. We also test the effectiveness of the burn-in procedure with these two examples. The computation in this subsection is performed on a high-performance computing cluster with an Intel Xeon Gold 6226R (16 Core) CPU and 3GB RAM.

\begin{example}
\label{example: convergence_test_single}
    Consider
    \begin{align*}
    V^{\textnormal{LE1}}(x_1, x_2) = \frac{x_1^2 + x_2^2}{2}, \quad b^{\textnormal{LE1}}(x_1, x_2) = (x_2, - x_1).
    \end{align*}
    This is the linearized version about $(0, 0)$ of Example \ref{example: 2D_single_well}. Following \cite{JPS17,raquepas2020large}, there is an open interval of values of $\alpha$ that contains $[0,1]$ and for which, for every $\varepsilon > 0$,
    \begin{align}
    \lambda^{\varepsilon, \alpha} = 1 - \sqrt{1 + 4 \alpha (1 - \alpha)}.
    \end{align}
\end{example}

\begin{example}
\label{example: convergence_test_double}
    Consider
    \begin{align*}
    V^\textnormal{LE2}(x_1, x_2) = -1 + 4(x_1-1)^2 + x_2^2, \quad b^\textnormal{LE2}(x_1, x_2) = (-x_2, x_1 - 1).
    \end{align*}
    This is the linearized version about $(1,0)$ of Example \ref{example: 2D_double_well} with $a=1$. Following \cite{JPS17,raquepas2020large}, there is an open interval of values of $\alpha$ that contains $[0,1]$ and for which, for every $\varepsilon > 0$,
    \begin{align}
    \lambda^{\varepsilon, \alpha} = \lambda_{+}^{\alpha},
    \end{align}
    where $\lambda_{+}^{\alpha}$ is given in \eqref{eq: eigenvalue_2D_double_well}. Note that $\lambda_{+}^{\alpha}$ is actually independent of $a$.
\end{example}

\begin{figure}[ht]
\centering
\begin{subfigure}[t]{0.31\textwidth}
\includegraphics[width=\columnwidth]{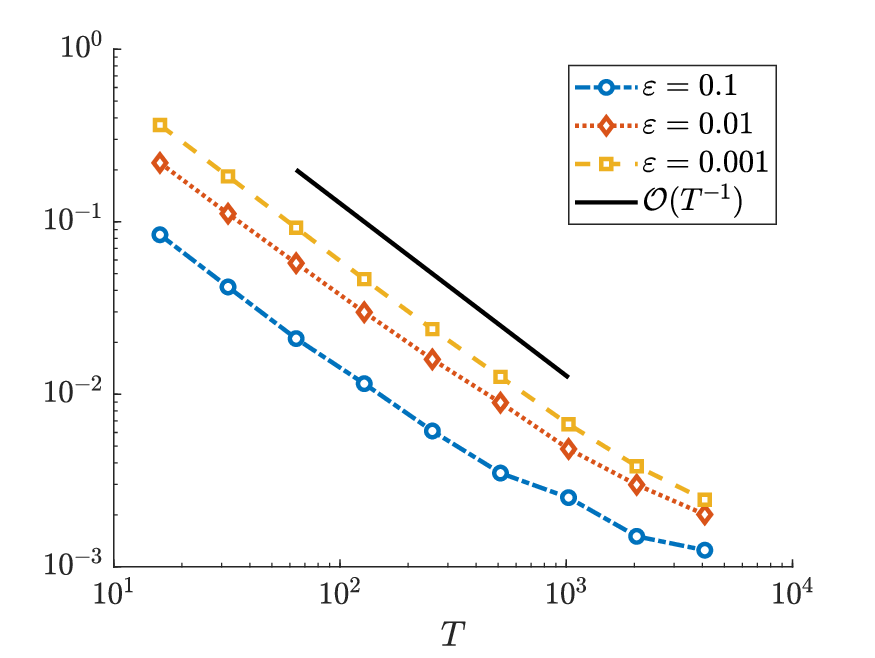}
\caption{Convergence test with respect to $T$ using $\Delta t = 2^{-7}$.}
\label{fig: convergence_test_single_T}
\end{subfigure}\hfill
\begin{subfigure}[t]{0.31\textwidth}
\includegraphics[width=\columnwidth]{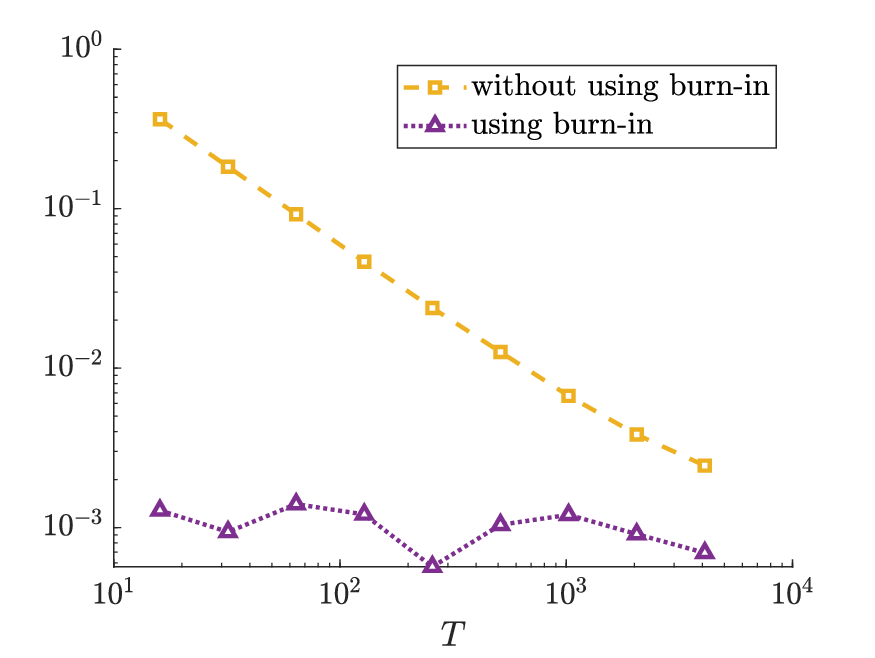}
\caption{Testing the effect of the burn-in procedure from $T/2$ on the convergence with respect to $T$, using $\varepsilon = 0.001$ and $\Delta t = 2^{-7}$.}
\label{fig: convergence_test_single_T_burn_in}
\end{subfigure}\hfill
\begin{subfigure}[t]{0.31\textwidth}
\includegraphics[width=\columnwidth]{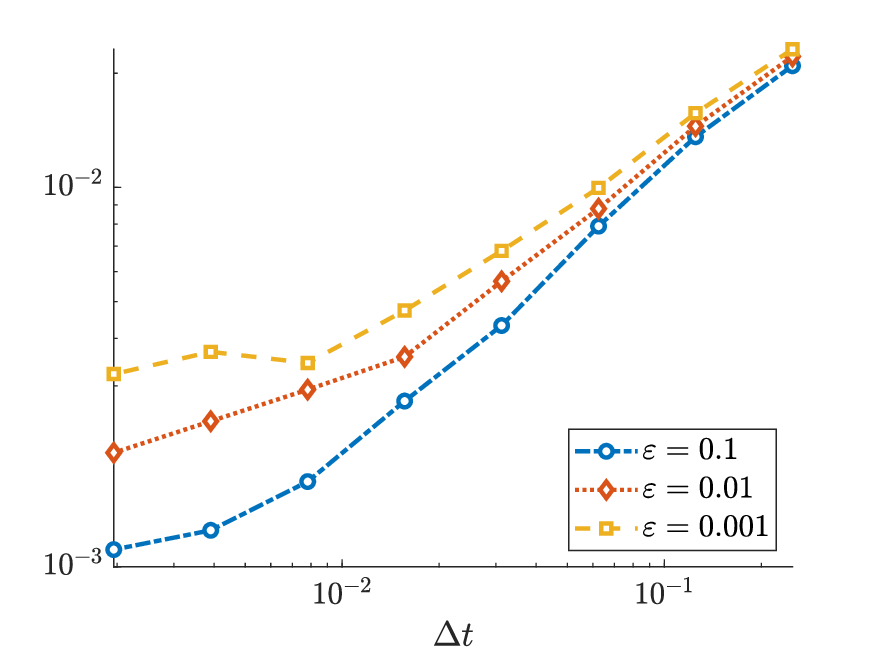}
\caption{Convergence test with respect to $\Delta t$ using $T = 2^{11}$.}
\label{fig: convergence_test_single_delta_t}
\end{subfigure}
\caption{In Example \ref{example: convergence_test_single}, we plot the error $| \hlambda^{\varepsilon,\alpha}_{\Delta t} - \lambda^{\varepsilon,\alpha}|$ at $\alpha = 0.25$ against $T$ and $\Delta t$ respectively, using $M = 500\,000$.}
\label{fig: convergence_test_single}
\end{figure}

\begin{figure}[ht]
\centering
\begin{subfigure}[t]{0.31\textwidth}
\includegraphics[width=\columnwidth]{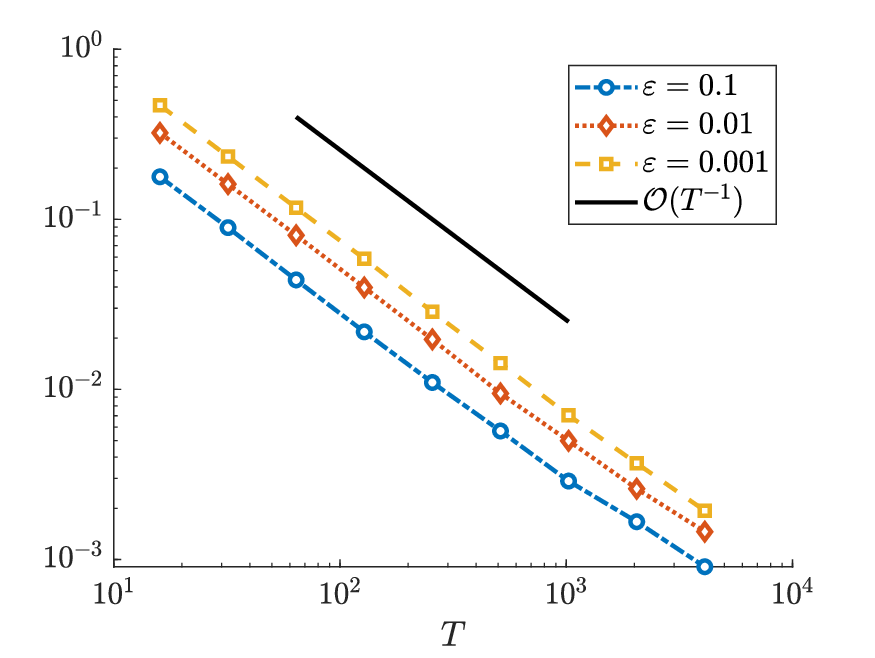}
\caption{Convergence test with respect to $T$ using $\Delta t = 2^{-7}$.}
\label{fig: convergence_test_double_T}
\end{subfigure}\hfill
\begin{subfigure}[t]{0.31\textwidth}
\includegraphics[width=\columnwidth]{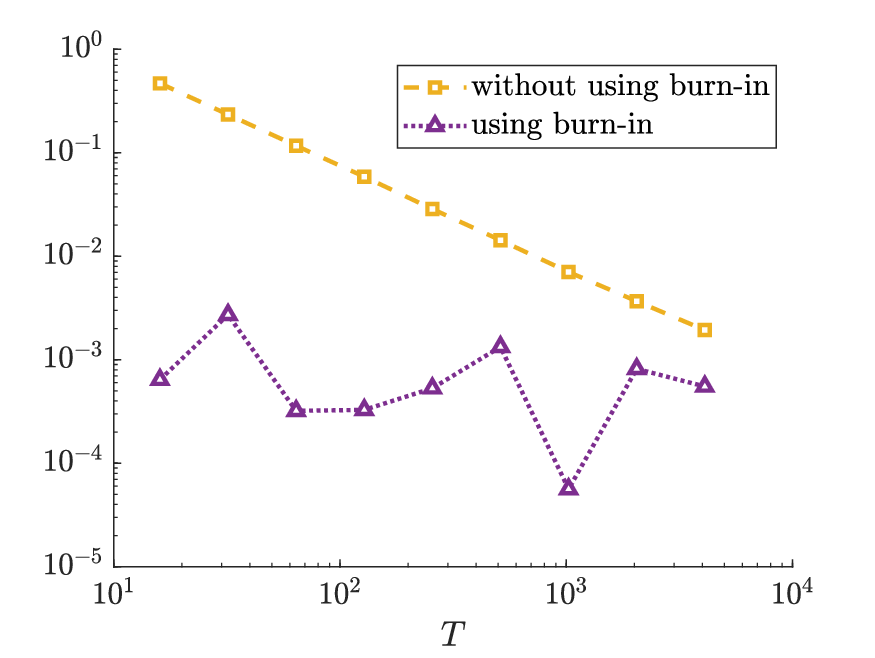}
\caption{Testing the effect of the burn-in procedure from $T/2$ on the convergence with respect to $T$, using $\varepsilon = 0.001$ and $\Delta t = 2^{-7}$.}
\label{fig: convergence_test_double_T_burn_in}
\end{subfigure}\hfill
\begin{subfigure}[t]{0.31\textwidth}
\includegraphics[width=\columnwidth]{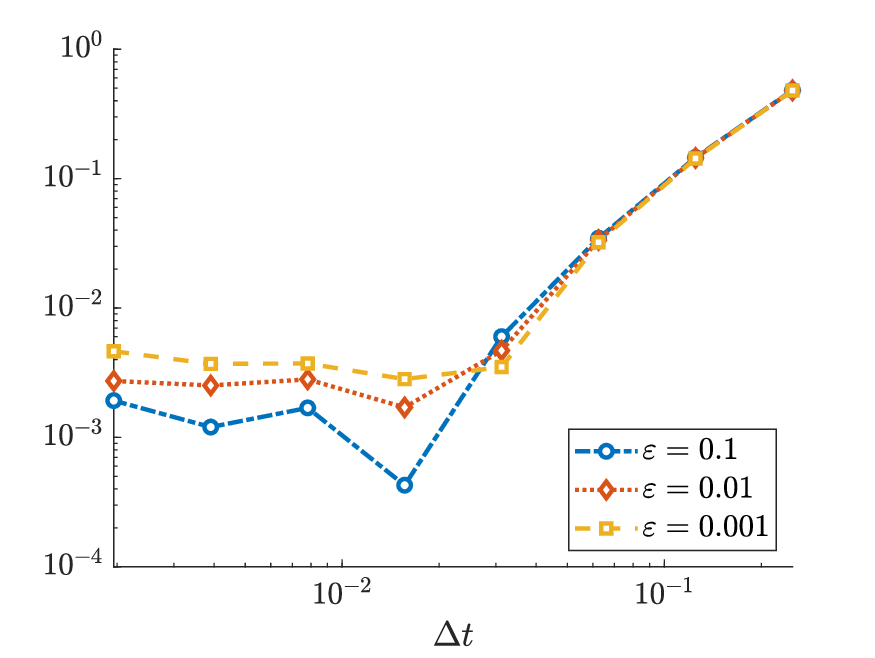}
\caption{Convergence test with respect to $\Delta t$ using $T = 2^{11}$.}
\label{fig: convergence_test_double_delta_t}
\end{subfigure}
\caption{In Example \ref{example: convergence_test_double}, we plot the error $| \hlambda^{\varepsilon,\alpha}_{\Delta t} - \lambda^{\varepsilon,\alpha}|$ at $\alpha = 0.25$ against $T$ and $\Delta t$ respectively, using $M = 500\,000$.}
\label{fig: convergence_test_double}
\end{figure}

We use the same numerical setting for the above two examples. We fix $\alpha = 0.25, M = 500\,000$ and choose the initial measure of the particles to be the standard multivariate Gaussian distribution. For the convergence test with respect to $T$, we fix $\Delta t = 2^{-7}$ and choose $T = 2^4, 2^5, \ldots, 2^{12}$ for each $\varepsilon = 0.1, 0.01, 0.001$. The error $| \hlambda^{\varepsilon, \alpha}_{\Delta t} - \lambda^{\varepsilon, \alpha} |$ is shown in Figures \ref{fig: convergence_test_single_T} and \ref{fig: convergence_test_double_T}. To test the effectiveness of the burn-in procedure, we fix $\varepsilon = 0.001, \Delta t = 2^{-7}$, choose $T = 2^4, 2^5, \ldots, 2^{12}$, and run the computation using the burn-in procedure, in which we start computing the eigenvalue from $t = \frac{T}{2}$. We show in Figures \ref{fig: convergence_test_single_T_burn_in} and \ref{fig: convergence_test_double_T_burn_in} the comparison between results obtained using and without using the burn-in procedure. For the convergence test with respect to $\Delta t$, we fix $T = 2048$ and choose $\Delta t = 2^{-2}, 2^{-3}, \ldots, 2^{-9}$ for each $\varepsilon = 0.1, 0.01, 0.001$. The error $| \hlambda^{\varepsilon, \alpha}_{\Delta t} - \lambda^{\varepsilon, \alpha} |$ is shown in Figures \ref{fig: convergence_test_single_delta_t} and \ref{fig: convergence_test_double_delta_t}.

As shown in the results of Examples \ref{example: convergence_test_single}--\ref{example: convergence_test_double}, the convergence rate with respect to $T$ is first order, and the burn-in procedure does help accelerate the computation, especially for small $T$. On the other hand, the convergence with respect to $\Delta t$ is more complicated. We are unable to identify the order of convergence in $\Delta t$, probably due to insufficiently large values of $M$ and $T$, but we are unable to afford a numerical setting with very large $T$ and $M$ due to hardware limitation. Nonetheless, the results suffice to confirm the convergence with respect to $\Delta t$.

In Figures \ref{fig: convergence_test_single_T} and \ref{fig: convergence_test_double_T}, we observe an increase in the error as $\varepsilon$ decreases. Several mechanisms could be at play for this increase: our (uninformed) standard Gaussian initialization of the particle is increasingly far from the invariant measure in Proposition \ref{prop: discrete stability}\footnote{In fact, in the examples considered, one can deduce by a change of variables that this invariant measure is still Gaussian, but with a variance that is rescaled by a factor of $\varepsilon$.}, affecting the multiplicative constant $C_\mu$; Arrhenius' law predicts very slow transitions between relevant critical points, etc.

\section{Conclusions}
\label{sec: conclusion}
We develop an interacting particle method for the computation of rate functions $I^\varepsilon$ for the large deviations of entropy production in the context of diffusion processes by equivalently computing the principal eigenvalue for a family of non-self-adjoint elliptic operators. We are particularly interested in the high-dimensional and vanishing-noise case, which is challenging to traditional numerical methods. We show that the principal eigenvalue can be well approximated in terms of the spectral radius of a discretized semigroup, making it suitable for an IPM. Moreover, we discuss two techniques for setting the initial measure in the IPM for faster computation. We present numerical examples of dimensions up to 16. The numerical results provide evidence that the numerical principal eigenvalue converges within visual tolerance to the analytical vanishing-noise limit with a fixed number of particles and a fixed time step size. Furthermore, the observed asymptotic behavior of the empirical density at the final time in the vanishing-noise limit is consistent with the theory in \cite{fleming1997asymptotics}. Our paper appears to be the first one to obtain numerical results of principal eigenvalue problems in such high dimensions. {Our method also allows us to probe the rate function $I^\varepsilon$ in situations where no explicit formulas are available, as well as to explore the gap between the theoretical works on different scalings for the vanishing-noise limit $\varepsilon \to 0^+$.}

In the future, it would be interesting to systematically investigate the error estimate of the IPM with respect to the numerical parameters of the method. Furthermore, the method should also be used to study the large deviation rate functions in situations that go beyond the scope of the theoretical works \cite{BDG15,BGL22,JPS17,raquepas2020large}, e.g.\ combining non-linearity of the vector field with the degeneracy of the noise. Finally, it should be noted that one could explore the possible benefits of considering higher-order schemes for SDEs (e.g.\ the Milstein method \cite{mil1975approximate} and high-order Runge--Kutta schemes (see e.g. \cite{rossler2009second})) or more sophisticated resampling procedures.

\begin{appendices}

\section{Proof sketches} 
\label{append: proofs}

\begin{proof}[Proof sketch of Proposition \ref{prop: continuous stability}]
We follow \cite[Section 2.3]{ferre2021more}. Picking $\theta$ sufficiently small such that $32\theta < \| H_0 \|^{-2}$ in Assumption \ref{assumption: quadratic growth of potential}, one can show that the growth bounds in Assumptions \ref{assumption: quadratic growth of potential}--\ref{assumption: bounded velocity} imply that $W$ is a Lyapunov function. The regularity properties in Assumptions \ref{assumption: quadratic growth of potential}--\ref{assumption: bounded velocity} can be used to show that the semigroup satisfies a Deoblin-type minorization property, an irreducibility property, and a local regularity property that then suffice to deduce \eqref{eq: convergence of conitinous Feynman--Kac semigroup}.

Since the semigroup is positivity preserving, compact and irreducible, for every fixed $t_0 > 0$, the (positive) spectral radius $\Lambda_{t_0}$ of $P_{t_0}^U$ is a simple eigenvalue, is related to the leading real eigenvalue $\lambda$ of the generator according to $\Lambda_{t_0} = \mathrm{e}^{t_0 \lambda}$, and admits a positive eigenvector $h$ with $\|h\|_{L^\infty_W} = 1$. One can also show that the eigenspace for $\lambda$ is actually 1-dimensional and is the only eigenspace admitting a positive eigenvector; we refer the reader to \cite[\S{V.2,\,VI.3,\,VI.3}]{engel2000one} for more details on the spectral theory of such semigroups (Perron--Frobenius-type theorems (a.k.a.\ Krein--Rutman-type theorems), spectral mapping theorems, etc.) and to \cite[App.\,A]{BDG15} and \cite[App.\,A]{raquepas2020large} for a treatment of the particular semigroups of interest using these tools. Hence, we can take $\varphi = h$ in \eqref{eq: convergence of conitinous Feynman--Kac semigroup} at times of the form $t = kt_0$, deduce that
\[
t_0 \lambda = \lim_{k\to\infty} \frac{1}{k} \log \pair{\mu}{P_{k t_0}^U  \1},
\]
and then pass to \eqref{eq: computation of eigenvalue in contiuous case} using standard arguments.
\end{proof}

\begin{proof}[Proof sketch of Proposition \ref{prop: discrete stability}]
    We follow \cite[Section 2.2]{ferre2021more}. The constant function $\1$ is a Lyapunov function for $\hP_{\Delta t}^U$. To see this, note that the action of the first operator in the splitting leaves $\1$ invariant and that the action of the second operator is such that
    \begin{align}
         \exp(U \Delta t)\1 \leq & \left(\sup_{|y| > R} \exp(\Delta t U(y)) \right)\1 + \left(\sup_{|y| \leq R} \exp(\Delta t U(y))\right) \1_{\{y : |y| \leq R\}},
    \end{align}
    and then take $R \to \infty$ using the growth bounds in Assumptions \ref{assumption: quadratic growth of potential}--\ref{assumption: bounded velocity}.
    The regularity properties in Assumptions \ref{assumption: quadratic growth of potential}--\ref{assumption: bounded velocity} can be used to show that $\hP_{\Delta t}^U$ satisfies a Deoblin-type minorization property, an irreducibility property, and a local regularity property that then suffice to deduce \eqref{eq: stability of Euler scheme} for some uniquely determined probability measure $\hmu_{\Delta t}^{\star}$ satisfying \begin{equation}
    \label{eq:hmu-is-invar}
        \Phi_{1,\Delta t}\hmu_{\Delta t}^{\star} = \hmu_{\Delta t}^{\star}.
    \end{equation}
    Moreover, one can show that the spectral radius $\hLambda_{\Delta t} $ for $\hP_{\Delta t}^U$ admits a positive eigenvector $\widehat{h}$ with $\|\widehat{h}\|_{L^\infty} = 1$, and that no other eigenvalue admits a positive eigenvector. Taking $\varphi = \widehat{h}$ in \eqref{eq: stability of Euler scheme}, one can deduce that
    \[
    \log \hLambda_{\Delta t}
        = \lim \limits_{k \rightarrow \infty} \frac{1}{k} \log \pair{\mu}{(\hP_{\Delta t}^U)^k\1}.
    \]
    Finally, since \eqref{eq:hmu-is-invar} implies in particular that $\pair{\Phi_{1,\Delta t}\hmu_{\Delta t}^{\star}}{\widehat{h}} = \pair{\hmu_{\Delta t}^{\star}}{\widehat{h}}$, it follows from the eigenvalue equation for $\widehat{h}$ and the definition of $\Phi_{1,\Delta t}$ that $\hLambda_{\Delta t} = \pair{\hmu_{U,\Delta t}^\star}{\hP_{\Delta t}^U\1}$.
\end{proof}

\begin{proof}[Proof sketch of Theorem \ref{thm: convergence of IPM wrt time step size}]
    Fix $T > 0$ and set
    \begin{align*}
        \widetilde{\Pi}_n := (\wP^U_{Tn^{-1}})^n
        \text{ and }
        \widehat{\Pi}_n := (\hP^U_{Tn^{-1}})^n.
    \end{align*}
    We show in four steps that the spectral radii of $\widetilde{\Pi}_n$ and $\widehat{\Pi}_n$ both converge to that of $P_T^U$ as operators on $C_0$ equipped with the $L^\infty$-norm.
    \begin{description}
        \item[Step 1.] \emph{Operator norm convergence $\|\widehat{\Pi}_n - \widetilde{\Pi}_n\| \to 0$.} The growth conditions on $V$ in Assumption \ref{assumption: quadratic growth of potential} and the control on $b$ in Assumption \ref{assumption: bounded velocity} imply that $\exp(TU)$ is bounded by $K := \exp(T\|U_+\|_{L^\infty})$
        and satisfies the following decay property: for every $\delta > 0$, there exists $R_\delta$ such that
        \[
            \sup_{|x| > R_\delta} \exp(TU(x)) < \delta.
        \]
        The boundedness in Assumption \ref{assumption: bounded velocity} allows for the application of a classical martingale argument that shows that, for every $\eta > 0$, there exists $\rho_\eta$ such that
        \begin{equation}
        \label{eq:not-past-rho}
            \sup_x \mathbb{P}^x\left\{\sup_{t \in [0,T]} |X_t^x - x| \geq \rho_\eta\right\} < \eta.
        \end{equation}

        One can show that $|\widehat{\Pi}_n\varphi(x) - \widetilde{\Pi}_n\varphi(x)|$ can be made arbitrarily small with large $n$, uniformly in $x$ and $\varphi$ with $\|\varphi\|_{L^\infty} = 1$ as follows. Choose $\delta$ and $\eta$ small enough, then $R_\delta$ and $\rho_\eta$ accordingly, and then consider separately the cases $|x| > R_\delta + \rho_\eta$ and $|x| \leq R_\delta + \rho_\eta$.
        The former will be small as is, and as for the latter, take $n$ large to leverage results for the Euler--Mayurama scheme in total variation norm \cite{BJ22}.

        \item[Step 2.] \emph{Strong convergence $\widetilde{\Pi}_n - P^U_T \stackrel{\text{s}}{\to} 0$.} Since we already know that $\mathcal{L} + U$ generates a strongly continuous semigroup on $C_0$, this is a direct consequence of Trotter's product formula for the semigroups generated by $\mathcal{L}$ and $U$ on that same space \cite{Tr59}.

        \item[Step 3.] \emph{Collective compactness of $(\widetilde{\Pi}_n)_{n=1}^\infty$.} We want to show that
        \[
            S := \{\widetilde{\Pi}_n \varphi : n \in \mathbb{N}, \varphi \in C_0, \|\varphi\|_{L^\infty} \leq 1 \}
        \]
        is precompact in $C_0$. To do this, we need to show three properties: boundedness, uniform vanishing at infinity, and equicontinuity.
        \begin{enumerate}
            \item[3a.] Pointwise, it follows from the definition of $\widetilde{\Pi}_n$ and the assumption that $\|\varphi\|_{L^\infty} \leq 1$ that
            \begin{align}
            \label{eq:Qn-pw-bound}
                |(\widetilde{\Pi}_n\varphi)(x)| \leq \mathbf{E}\Bigg[ \exp \left( \sum_{k=0}^{n-1} \frac{T}{n} U(X_{kTn^{-1}}) \right) \Bigg| X_0 = x \Bigg]
            \end{align}
            This shows, among other things, that the family $S$ is bounded in norm by $K$.

            \item[3b.] Using \eqref{eq:not-past-rho} in conjunction with properties of $U$ discussed in Step 1, we see that the expectation on the right-hand side of \eqref{eq:Qn-pw-bound} is arbitrarily small, simultaneously for all $n$ and $\varphi$, as soon as $x$ is outside of a sufficiently large ball. Hence, the family $S$ does uniformly vanish at infinity.

            \item[3c.] To show equicontinuity, we consider the differences
            \begin{align*}
                \widetilde{\Pi}_n\varphi(x) - \widetilde{\Pi}_n\varphi(y) = & \mathbf{E}^x\Bigg[ \varphi(X_T) \exp \left( \sum_{k=0}^{n-1} \frac{T}{n} U(X_{kTn^{-1}}) \right) \Bigg] \\
                & - \mathbf{E}^y\Bigg[\varphi(Y_T) \exp \left( \sum_{k=0}^{n-1} \frac{T}{n} U(Y_{kTn^{-1}}) \right) \Bigg]
            \end{align*}
            with $y$ close to some fixed $x$\,---\,we require $|y-x| < 1$ to begin. Note that this difference of expectations can be computed by realizing the two processes $(X_t)_{t \in [0,T]}$ and $(Y_t)_{t \in [0,T]}$ on a common probability space as we see fit\,---\,this is the classical coupling method; see e.g. \cite{Th95} and historical references therein.

            By continuity, the difference
            \[
                \delta_2 := |U(x) - U(y)|
            \]
            can be made arbitrarily small by taking $y$ close enough to $x$.
            In view of the ultra-Feller property \cite{Ha09}, given any $\tau > 0$, the difference
            \[
               \delta_3(\tau) := \|P_\tau^*\delta_x- P_\tau^*\delta_y\|_\text{TV}
            \]
            can be made arbitrarily small by taking $y$ close enough to $x$. Hence, once such a $\tau$ is given, we can realize the two processes $(X_t)_{t \in [0,T]}$ and $(Y_t)_{t \in [0,T]}$ on a common probability space so that
            \[
                \mathbb{P}^{(x,y)}\{ X_t \neq Y_t \text{ for some } t \in [\tau, T]\} \leq \delta_3(\tau).
            \]
            All in all, we have
            \begin{align*}
                |\widetilde{\Pi}_n\varphi(x) - \widetilde{\Pi}_n\varphi(y)| \leq &
                \mathbf{E}^{(x,y)}\Bigg| \varphi(X_T) \exp \left( \sum_{k=0}^{n-1} \frac{T}{n} U(X_{kTn^{-1}}) \right) - \varphi(Y_T) \exp \left( \sum_{k=0}^{n-1} \frac{T}{n} U(Y_{kTn^{-1}}) \right) \Bigg| \\
                \leq & 2 K\eta + 2 K\delta_3(\tau)
                + KT\delta_2 + 4K\tau \sup_{|z-x|<1+\rho_\eta} |U(z)|.
            \end{align*}
            This can be made arbitrarily small, uniformly in $\varphi$ and $n$, as follows. First, we choose $\eta$ so that the first term is as small as desired, and then we fix $\rho_\eta$ accordingly. Next, we take $\tau$ small enough so that the last term is as small as desired. Finally, once $\tau$ is fixed, we can choose a coupling to compute the expectation, and the second and third terms will be as small as desired as long as $y$ is close enough to $x$.
        \end{enumerate}

        \item[Step 4.] \emph{Spectral theory.}  On the one hand, Step 1 and the fact that both sequences of operators are uniformly bounded by $K$ ensures that $|\operatorname{spr}(\widetilde{\Pi}_n) - \operatorname{spr}(\widehat{\Pi}_n)| \to 0$ by classical perturbation theory arguments; see e.g. \cite{Kat}. On the other hand, thanks to the spectral analysis of \cite{AP68} for collectively compact sequences of operators that converge strongly, Steps 2 and 3 show that $|\operatorname{spr}(\widetilde{\Pi}_n) - \operatorname{spr}(P_T^U)| \to 0$.
    \end{description}
    Clearly, if $\wLambda_{Tn^{-1}}$ is an eigenvalue of $\wP_{Tn^{-1}}^U$ with a positive eigenvector, then $(\wLambda_{Tn^{-1}})^n$ is an eigenvalue of $\widetilde{\Pi}_n = (\wP_{Tn^{-1}}^U)^n$ with a positive eigenvector. Hence, the identity $\operatorname{spr}(\widetilde{\Pi}_n) = (\wLambda_{Tn^{-1}})^n$ follows from the fact that the spectral radius is the only eigenvalue admitting a positive eigenvector. Similarly, $\operatorname{spr}(\widehat{\Pi}_n) = (\hLambda_{Tn^{-1}})^n$. Finally, the fact that $\operatorname{spr}(P_T^U) = \exp(\lambda T)$ is a consequence of the spectral mapping theorem \cite{engel2000one}, so the proof is completed.
\end{proof}

\end{appendices}

\paragraph*{Acknowledgements}

R.R. was partially funded by the \emph{Fonds de recherche du Qu\'ebec\,---\,Nature et technologies} (FRQNT) and by the Natural Sciences and Engineering Research Council of Canada (NSERC). J.X. was partially supported by NSF grant DMS-2309520. Z.Z. was supported by the National Natural Science Foundation of China (Projects 92470103 and 12171406), Hong Kong RGC grant (Projects 17307921, 17304324, and 17300325), Seed Funding Programme for Basic Research (HKU), the Outstanding Young Researcher Award of HKU (2020--21), Seed Funding for Strategic Interdisciplinary Research Scheme 2021/22 (HKU), and an R\&D Funding Scheme from the HKU-SCF FinTech Academy. The project was initiated at the Courant Institute, New York University, where J.X. was visiting in the Fall of 2022. The authors would like to thank Professors R.\ Caflisch, R.\ Kohn, D.\ McLaughlin, C.\ Peskin, S.\ R.\ S.\ Varadhan, and L.-S.\ Young for helpful conversations and scientific activities that made our collaboration possible. The computations were performed at the research computing facilities provided by Information Technology Services, the University of Hong Kong.

\bibliographystyle{siam}
\bibliography{reference}
\end{document}